\documentclass{amsart}
\usepackage[pdftex]{graphicx}

\usepackage{mathtools}
\usepackage{afterpage}
\usepackage{enumitem}
\usepackage{multirow}

\usepackage{bm}
\usepackage{color}
\usepackage{caption} 
\mathtoolsset{showonlyrefs=true}

\numberwithin{equation}{section}

\newtheorem{theorem}{Theorem}[section]
\newtheorem{lemma}[theorem]{Lemma}

\newtheorem{proposition}[theorem]{Proposition}

\newtheorem*{question}{Question}

\theoremstyle{definition}
\newtheorem{definition}[theorem]{Definition}

\newtheorem*{acknowledgment}{Acknowledgments}
\newtheorem*{organization}{Organization}

\theoremstyle{remark}
\newtheorem{remark}[theorem]{Remark}

\numberwithin{equation}{section}

\newcommand{\Z}{\mathbb{Z}}
\newcommand{\R}{\mathbb{R}}
\newcommand{\C}{\mathbb{C}}
\newcommand{\F}{\mathbb{F}}
\newcommand{\CFK}{{\rm CFK}}
\newcommand{\HFK}{{\rm HFK}}

\begin{document}
%%%%%%%%%%%%%%%%%%%%% Publisher's Area please ignore %%%%%%%%%%%%%%
%\catchline{}{}{}{}{}
%%%%%%%%%%%%%%%%%%%%%%%%%%%%%%%%%%%%%%%%%%%%%%%%%%%%%%%%%%%%

\title[Hyperbolic knots whose upsilon invariants are convex]{New family of hyperbolic knots whose upsilon invariants are convex}

\author{Keisuke Himeno
%\footnote{Typeset names in 8~pt Times Roman, uppercase
%and lightface.  Use footnotes only to indicate if permanent and
%present addresses are different. Funding information should go
%in the Acknowledgement section.}
}

\address{Graduate School of Advanced Science and Engineering, Hiroshima University,
1-3-1 Kagamiyama, Higashi-hiroshima, 7398526, Japan}
\email{himeno-keisuke@hiroshima-u.ac.jp}

\begin{abstract}
The  Upsilon invariant of a knot is a concordance invariant derived from knot Floer homology theory. It is a piecewise linear continuous function defined on the interval $[0,2]$. Borodzik and Hedden gave a question asking for which knots the Upsilon invariant is a convex function. It is known that 
the Upsilon invariant of any $L$--space knot, and a Floer thin knot after taking its mirror image, if necessary,  as well, is convex. Also, we can make infinitely many knots whose Upsilon invariants are convex by the connected sum operation.
 In this paper, we construct infinitely many mutually non-concordant hyperbolic knots whose Upsilon invariants are convex. To calculate the full knot Floer complex, we make use of a combinatorial method for  $(1,1)$--knots.
\end{abstract}

\renewcommand{\thefootnote}{}
%affiliation and mailing addresses in 8pt Times italic.} \\

%%%%%%%%%%%%%%%%%%%%%%%%%%%%%%%%%%%%%%%%%%%%%%%%%%%%%%%%%%%%%%%%%%%%%%%%%%%%%%%%%%%%%%%%%%%%%%%%%%%%%%%%
\maketitle

%%%%%%%%%%%%%%%%

\section{Introduction}\label{intro}
For a knot $K\subset S^3$, the {\it Upsilon invariant\/} $\Upsilon_K(t)\colon [0,2]\to\R$ is a concordance invariant introduced within the framework of knot Floer homology theory \cite{OSS}.  It has the following features:
\begin{itemize}
\item $\Upsilon_K(t)$ is continuous and piecewise liner;
\item $\Upsilon_K(t)=\Upsilon_K(2-t)$;
\item $\Upsilon_{-K}(t)=-\Upsilon_K(t)$ where $-K$ is the mirror image of $K$ with reversed orientation;
\item $\Upsilon_K(t)$ is additive under the connected sum operation, that is, $\Upsilon_{K\#L}(t)=\Upsilon_K(t)+\Upsilon_L(t)$;
\item for small $t$, $\Upsilon_K(t)=-\tau(K)\cdot t$ where $\tau(K)$ is the {\it tau invariant} of $K$ (see \cite{OStau} for the tau invariant).
\end{itemize}
It was originally defined using the $t$--modified knot Floer complex, but Livingston gave an interpretation on the {\it full knot Floer complex\/} $\CFK^{\infty}(K)$ \cite{Li}. This made it relatively easy to calculate the Upsilon invariant. 

In \cite{BH}, Borodzik and Hedden showed that the Upsilon invariant of any $L$--space knot is a convex function. In view of this, they gave the following question.

\begin{question}[Question 1.4 of \cite{BH}]
For which knots is $\Upsilon_K$ a convex function?
\end{question}

Clearly, $\Upsilon_{K_1\#K_2}$ is a convex function if both $\Upsilon_{K_1}$ and $\Upsilon_{K_2}$ are convex. Also, for an alternating knot, more generally, a Floer thin knot, the Upsilon invariant is given as $\Upsilon_K(t)=-\tau(K)(1-|1-t|)$ \cite{Al}. Therefore, this is convex after taking  an appropriate mirror image.  

The purpose of this paper is to give infinitely many mutually non-concordant hyperbolic knots which
provide new answers to Borodzik and Hedden's question.

 Let $n$ be a non-negative integer, and let $q\ge 4$ be an integer coprime to $3$. 
 Then $q=3k+1$ or $q=3k+2$ for some $k>0$. A knot $K_n^{(3,q)}$ is defined as in Figure \ref{knot3k+1} when $q=3k+1$, otherwise as in Figure \ref{knot3k+2}. Note that $K_0^{(3,q)}$ is the torus knot $T(3,q)$ of type $(3,q)$.

\begin{figure}[h]
\centering
\includegraphics[width=0.6\linewidth]{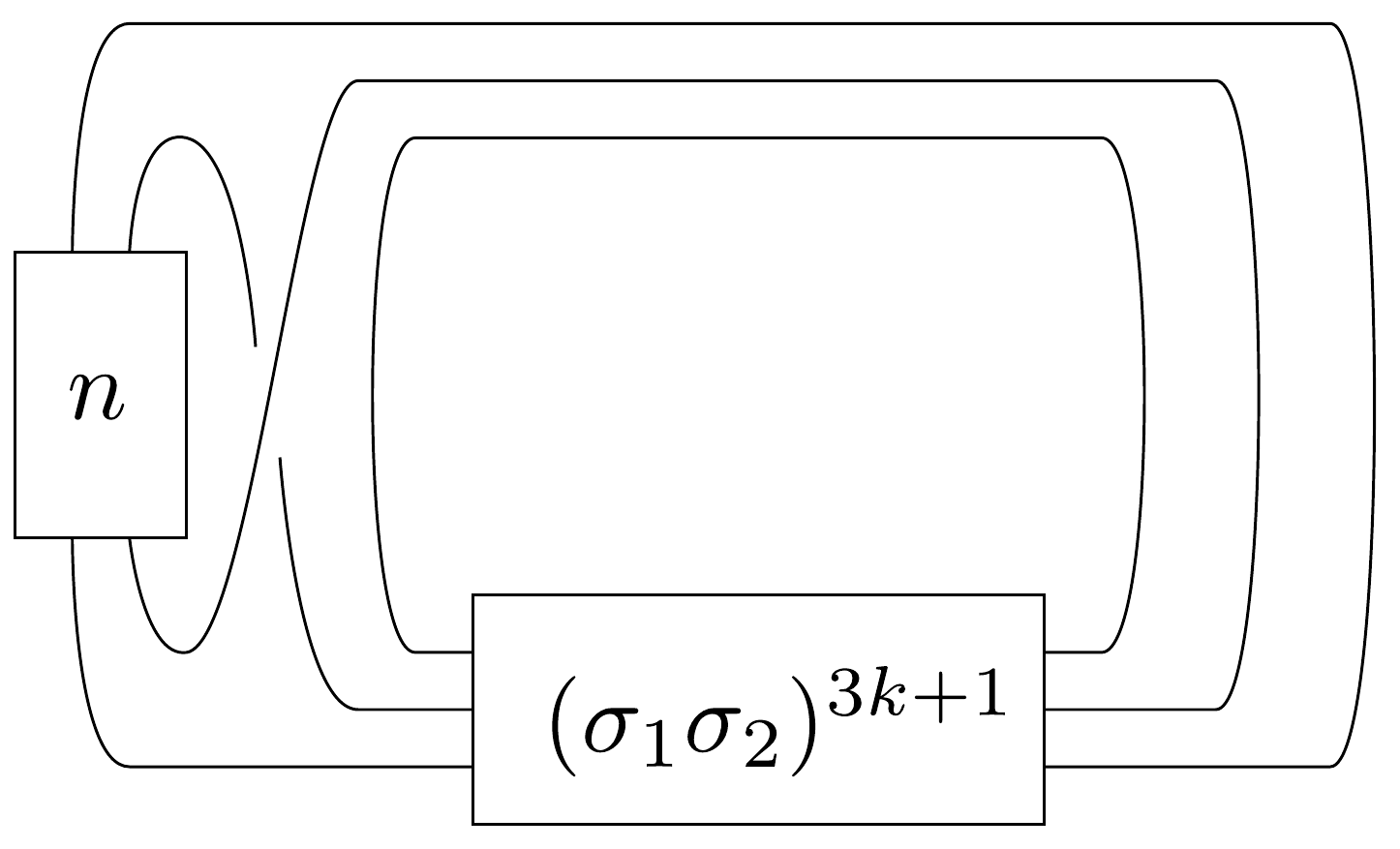}  
\caption{The case where $q=3k+1$. $\sigma_1$ and $\sigma_2$ are standard generators of the  $3$-braid group. The box with number $n$ indicates $n$ right handed vertical full-twists.}
\label{knot3k+1}
\end{figure}

\begin{figure}[h]
\centering
\includegraphics[width=0.6\linewidth]{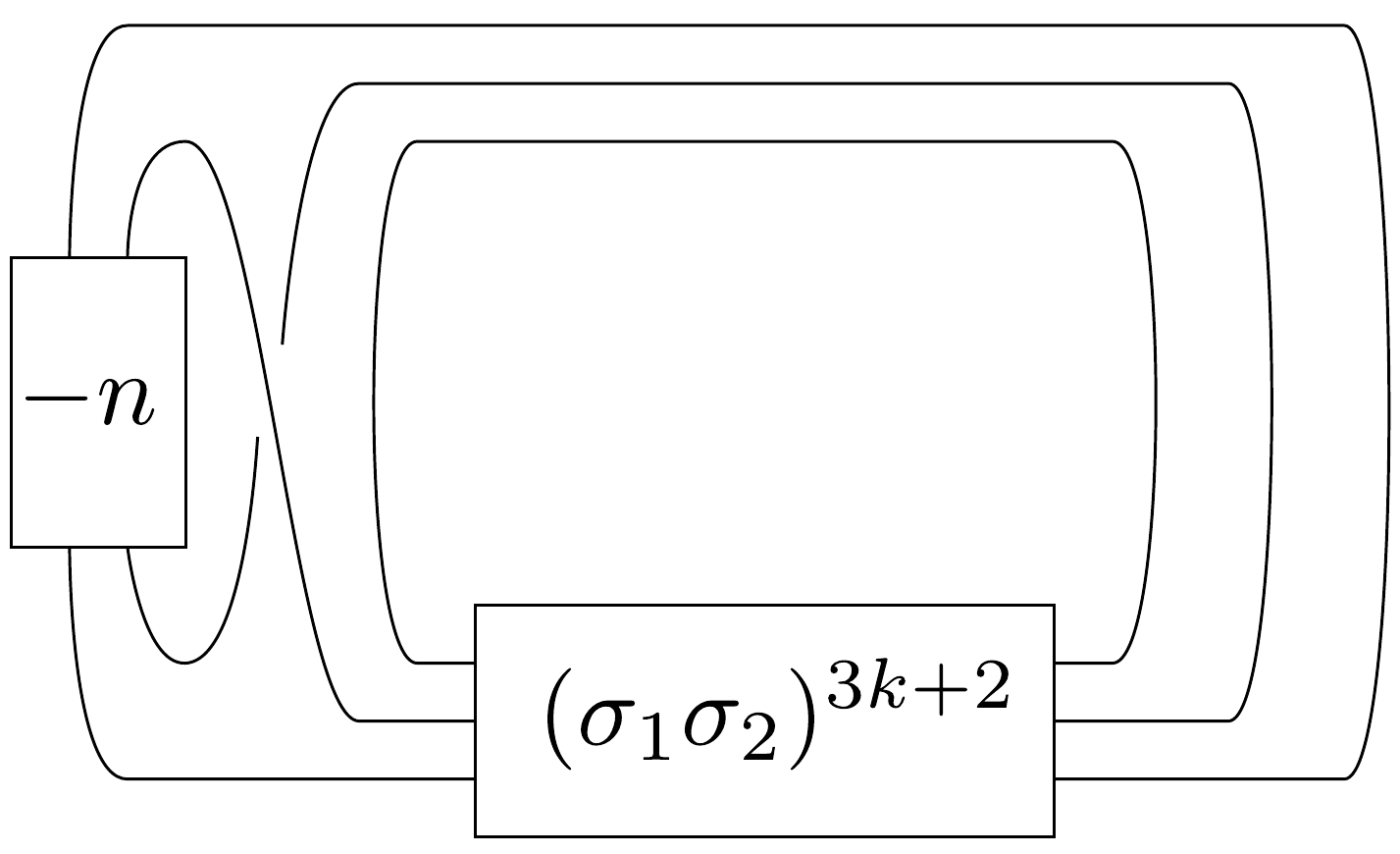}  
\caption{The case where $q=3k+2$. The box with number $-n$ indicates $n$ left handed vertical full-twists.}
\label{knot3k+2}
\end{figure}

\begin{theorem}\label{mainthm1}
For $n\ge1$, $K_n^{(3,q)}$ satisfies the following\textup{:}
\begin{enumerate}
\setlength{\itemsep}{2mm}
\item[\textup{(1)}] $\Upsilon_{K_n^{(3,q)}}$ is a convex function.
\item[\textup{(2)}] $K_n^{(3,q)}$ is a hyperbolic knot.
\item[\textup{(3)}] $K_n^{(3,q)}$ is neither an $L$--space knot nor a Floer thin knot.
\end{enumerate}
\end{theorem}

Two knots $K_n^{(3,q)}$ and $K_m^{(3,q')}$ are not equivalent when $n\ne m$ or $q\ne q'$ (this can be seen from their Alexander polynomials, see Lemma \ref{AlexanderPoly}). 

Since the Upsilon invariant is a concordance invariant, if $K$ is concordant to a knot whose Upsilon invariant is convex, $\Upsilon_K$ is also convex. Hence, it makes sense to construct mutually non-concordant knots.

\begin{theorem}\label{mainthm2}
For a fixed integer $q$, the family $\{K_n^{(3,q)}\}_{n=0}^{\infty}$ contains infinitely many mutually non-concordant knots. 
\end{theorem}

 \begin{organization}
The paper is organized as follows. In Section \ref{fullknotFloercomplex}, we review the properties of the full knot Floer complex and give the definition of stable equivalence. In Section \ref{ProofThm1}, we admit the calculation of the full knot Floer complex of our knots and then prove Theorem \ref{mainthm1}, and in Section \ref{ProofThm2}, we discuss the concordance relation among our knots to prove Theorem \ref{mainthm2}. In Section \ref{(1,1)diagram}, we give a $(1,1)$--diagram of $K^{(3,3k+1)}_n$, and in Section \ref{calCFK3k+1}, we calculate the full knot Floer complex of $K^{(3,3k+1)}_n$ from its $(1,1)$--diagram. Lastly, in Section \ref{calCFK3k+2}, we give the full knot Floer complex of $K_n^{(3,3k+2)}$ in the same manner as in Sections \ref{(1,1)diagram} and \ref{calCFK3k+1}.
 \end{organization}

 \begin{acknowledgment}
The author would like to thank Masakazu Teragaito for his thoughtful guidance and helpful discussions about this work. This work was supported by Japan Science and Technology Agency (JST), the establishment of university fellowships towards the creation of science technology innovation, Grant Number JPMJFS2129.
\end{acknowledgment}
%%%%%%%%%%%%%%%%%%%%%%%%%%%%%%%%%%%%%%%%%%%%%%%%%%%%%%%%%%%%%%%%%%%%%%%%%%%%%%%%%%%%%%%%%%%%%%%%%%%%%%%%%%%%%%%%
\section{The full knot Floer complex $\CFK^{\infty}(K)$}\label{fullknotFloercomplex}
We give a brief overview of the full knot Floer complex. For more details, see \cite{Ho,Ma,OS}. 

For a knot $K\subset S^3$, we can define the {\it full knot Floer complex\/} $\CFK^{\infty}(K)$, which has the following structure:

\begin{itemize}
\item $\CFK^{\infty}(K)$ is a finitely generated $\F_2[U,U^{-1}]$--module, where $\F_2$ is the finite field of two elements and $U$ is a formal variable. In Ozsv\'ath and Szab\'o's notation \cite{OS}, as an $\F_2$--module, $\CFK^{\infty}$ is generated by elements of the form $[x,i,j]$ equipped with the {\it Alexander grading} $A(x)=j-i\in\Z$. Moreover, $(i,j)$ grading induces the $\Z\oplus\Z$--filtration structure.
\item As a chain complex, $\CFK^{\infty}(K)$ has an integer valued homological grading, 
called the {\it Maslov grading}, and the boundary map $\partial$.
\item The action of $U$ on $\CFK^{\infty}(K)$ commutes with $\partial$, 
lowers Maslov grading by $2$, 
and lowers $(i,j)$ grading by $1$, that is, $U[x,i,j]=[x,i-1,j-1]$.
\end{itemize}

We can draw a structure of $\CFK^{\infty}(K)$ on a plane; assign a generator $[x,i,j]$ to the point which has coordinates $(i,j)$ on the plane, and the differentials are drawn by arrows (it can be represented by a line segment instead of an arrow, since the differential always lowers filtration levels). For example, see Figure \ref{CFK-T3,4}. 

\begin{figure}[h]
\centering
\includegraphics[scale=0.8]{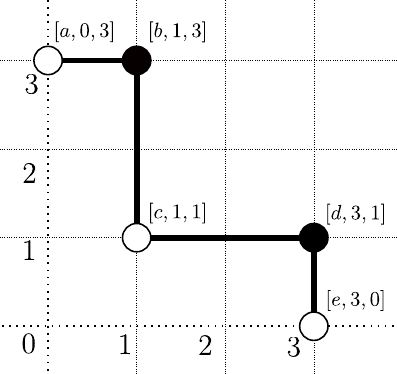}
\caption{The full knot Floer complex $\CFK^{\infty}(T(3,4))$ of $T(3,4)$. As an $\F_2[U,U^{-1}]$--module, there are five generators (vertices). Actually, by the action of $U$, complexes of the same shape are lined up in the upper right and lower left. However, we draw only the complex which has generators with Maslov grading $0$ (white vertices). Also, the differentials are represented by line segments, for example, $\partial[b,1,3]=[a,0,3]+[c,1,1]$.}
\label{CFK-T3,4}
\end{figure}

To prove (1) of Theorem \ref{mainthm1}, we introduce an equivalent relation among full knot Floer complexes.

\begin{definition}
Two full knot Floer complexes $C_1$ and $C_2$ are said to be {\it stably equivalent\/} if there are graded, $\Z\oplus\Z$--filtered, acyclic chain complexes $A_1$, $A_2$ such that $C_1\oplus A_1$ and $C_2\oplus A_2$ are filtered chain homotopy equivalent. 
\end{definition}

In fact,  an acyclic complex appeared in this paper is a {\it box complex\/}, see Figure \ref{boxcomplex}.

\begin{figure}[h]
\centering
\includegraphics[scale=0.4]{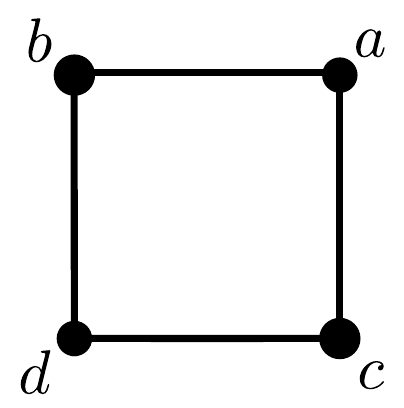}
\caption{A box complex. The homological cycle is $d$ or $b+c$. But both are boundary cycles, so this complex is acyclic.}
\label{boxcomplex}
\end{figure}

By the manner of calculating $\Upsilon_K$ in \cite{Li}, we can see that $\Upsilon_K$ is an invariant on stably equivalence classes:

\begin{proposition}[\cite{Ho,Li}]\label{stablyequivalentupsilon}
If $\CFK^{\infty}(K_1)$ and $\CFK^{\infty}(K_2)$ are stably equivalent, then $\Upsilon_{K_1}=\Upsilon_{K_2}$.
\end{proposition}

In fact,  Hom  \cite{Ho} shows that if $K_1$ and $K_2$ are (smoothly) concordant, then $\CFK^{\infty}(K_1)$ and $\CFK^{\infty}(K_2)$ are stably equivalent. Thus the Upsilon invariant (and others such as tau invariant and epsilon invariant) being a concordance invariant can be seen as a corollary of this theorem and Proposition \ref{stablyequivalentupsilon}.

%%%%%%%%%%%%%%%%%%%%%%%%%%%%%%%%%%%%%%%%%%%%%%%%%%%%%%%%%%%%%%%%%%%%%%%%%%%%%%%%%%%%%%%%%%%%%%%%%%%%%%%%%%%%%%%%
\section{Proof of Theorem \ref{mainthm1}}\label{ProofThm1}

In this section, we prove Theorem \ref{mainthm1}, admitting the calculation $\CFK^{\infty}(K_n^{(3,q)})$. Figures \ref{CFK-K_n3,3k+1} and \ref{CFK-K_n3,3k+2} show $\CFK^{\infty}(K_n^{(3,3k+1)})$ and $\CFK^{\infty}(K_n^{(3,3k+2)})$, respectively.
The calculation will be done in Sections \ref{calCFK3k+1} and \ref{calCFK3k+2}. 

\begin{figure}[h]
\centering
\includegraphics[scale=0.4]{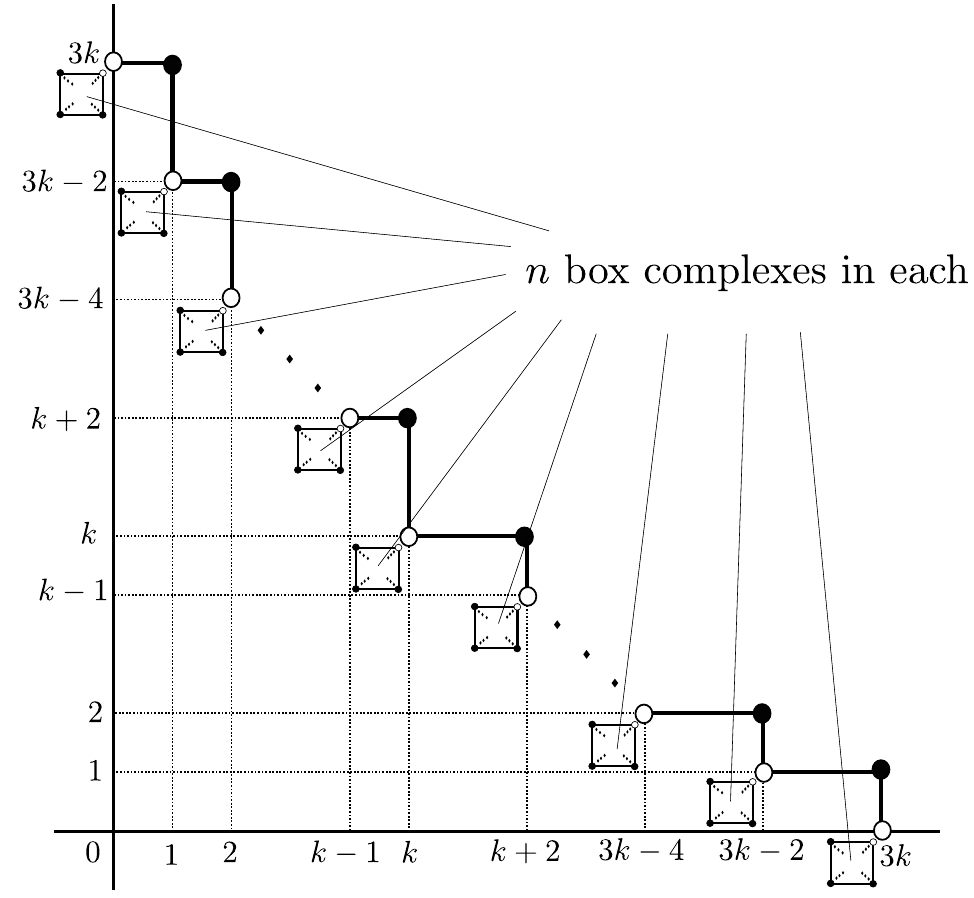}
\caption{The full knot Floer complex $\CFK^{\infty}(K_n^{(3,3k+1)})$. This complex consists of the staircase complex, which is consistent with the complex of $\CFK^{\infty}(T(3,3k+1))$, and box complexes. The vertices of box complexes are actually on the grid, but are drawn slightly displaced.}
\label{CFK-K_n3,3k+1}
\end{figure}

\begin{figure}[h]
\centering
\includegraphics[scale=0.4]{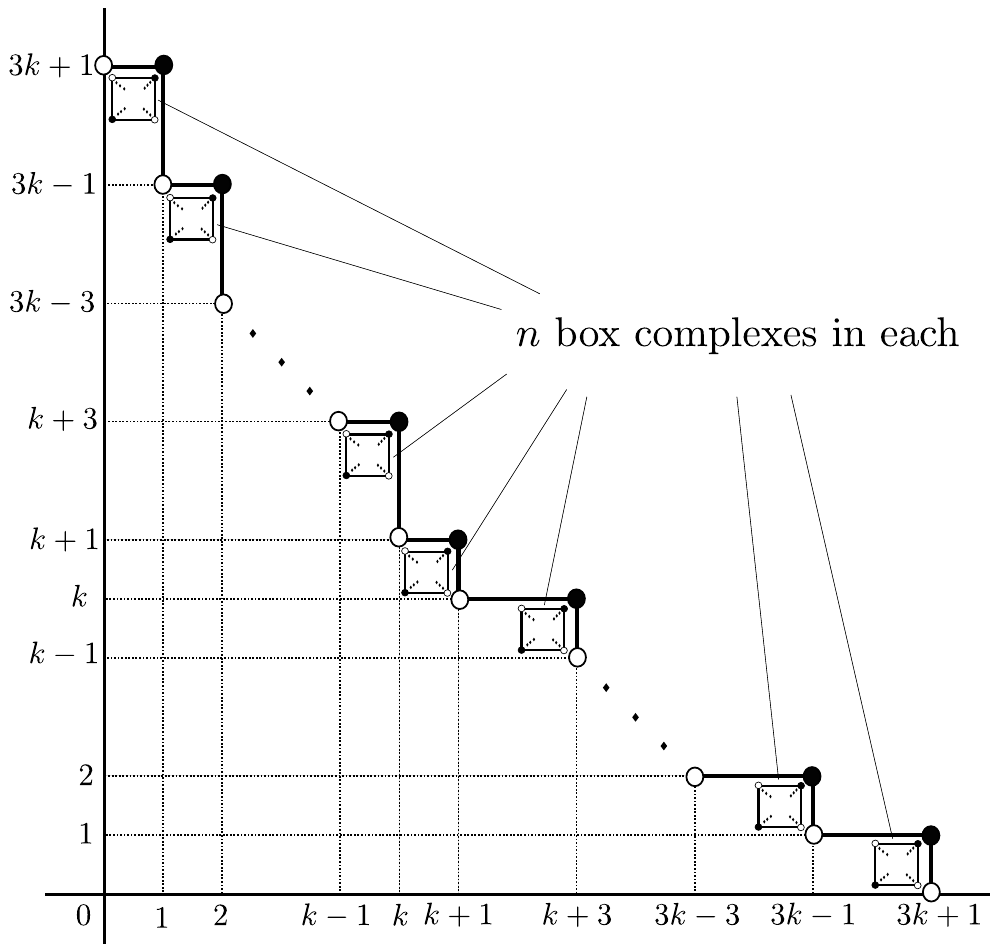}
\caption{The full knot Floer complex $\CFK^{\infty}(K_n^{(3,3k+2)})$. The staircase complex is consistent with $\CFK^{\infty}(T(3,3k+2))$}
\label{CFK-K_n3,3k+2}
\end{figure}

\begin{lemma}\label{stablyequivalent}
$\CFK^{\infty}(K_n^{(3,q)})$ is stably equivalent to $\CFK^{\infty}(T(3,q))$.
\end{lemma}
\begin{proof}
Since the positive torus knot is an $L$--space knot, its full knot Floer complex can be determined by the Alexander polynomial (see  Section 4.1 of \cite{BL}). According to this, we have $\CFK^{\infty}(T(3,q))$ as the staircase complex as shown in Figures \ref{CFK-K_n3,3k+1} or \ref{CFK-K_n3,3k+2}. Therefore, $\CFK^{\infty}(K_n^{(3,q)})$ is stably equivalent to $\CFK^{\infty}(T(3,q))$.
\end{proof}

\begin{proof}[Proof of (1) in Theorem \ref{mainthm1}]
By Proposition \ref{stablyequivalentupsilon} and Lemma \ref{stablyequivalent}, $\Upsilon_{K_n^{(3,q)}}=\Upsilon_{T(3,q)}$. Since $T(3,q)$ is an $L$--space knot, $\Upsilon_{T(3,q)}$ is a convex function (\cite{BH}), so is $\Upsilon_{K_n^{(3,q)}}$.
\end{proof}

To prove (2) and (3) of Theorem \ref{mainthm1}, it is useful to consider the {\it hat version knot Floer homology}. For a knot $K$, the hat version knot Floer homology $\widehat{HFK}_d(K;i)\ (d,i\in\Z)$ is a bi-graded Abelian group. It can be easily obtained from $\CFK^{\infty}(K)$ (for example, \cite{OS},\cite{Ma}). By $\CFK^{\infty}(K_n^{(3,q)})$ as shown in Figures \ref{CFK-K_n3,3k+1} and \ref{CFK-K_n3,3k+2}, we have the following lemma:

\begin{lemma}\label{hatversion}
The rank of \textup{(}hat version\textup{)} knot Floer homology of $K_n^{(3,q)}$ is as follows\textup{:}
\[
{\rm rank}\ \widehat{\HFK}_d(K_n^{(3,3k+1)};i)=\left\{
\begin{array}{lll}
n & (d,i)=(1-2j,3k+1-3j) & (j=0,\ldots,k)\\
n & (d,i)=(-2k-1-4j,-1-3j) & (j=0,\ldots,k)\\
2n+1 & (d,i)=(-2j,3k-3j) & (j=0,\ldots,k)\\
2n+1 & (d,i)=(-2k-4-4j,-3-3j) & (j=0,\ldots,k-1)\\
n+1 & (d,i)=(-1-2j,3k-1-3j) & (j=0,\ldots,k-1)\\
n+1 & (d,i)=(-2k-3-4j,-2-3j) & (j=0,\ldots,k-1)\\
0 & \text{otherwise},
\end{array}
\right.
\]
\[
{\rm rank}\ \widehat{\HFK}_d(K_n^{(3,3k+2)};i)=\left\{
\begin{array}{lll}
n+1 & (d,i)=(-2j,3k+1-3j) & (j=0,\ldots,k)\\
n+1 & (d,i)=(-2k-2-4j,-1-3j) & (j=0,\ldots,k)\\
2n+1 & (d,i)=(-1-2j,3k-3j) & (j=0,\ldots,k)\\
2n+1 & (d,i)=(-2k-5-4j,-3-3j) & (j=0,\ldots,k-1)\\
n & (d,i)=(-2-2j,3k-1-3j) & (j=0,\ldots,k-1)\\
n & (d,i)=(-2k-4-4j,-2-3j) & (j=0,\ldots,k-1)\\
0 & \text{otherwise}. &
\end{array}
\right.
\]
\end{lemma}

For a knot $K$, $\displaystyle{\Delta_K(t)=\sum_{d,i}(-1)^d t^i\cdot {\rm rank}\ \widehat{{\rm HFK}}_d(K;i)}$ \cite{OS}. By Lemma \ref{hatversion}, we can determine the Alexander polynomial of our knots:

\begin{lemma}\label{AlexanderPoly}
The Alexander polynomial of $K_n^{(3,3k+1)}$ is given as 
\begin{align*}
\Delta_{K_n^{(3,3k+1)}}(t)=&\sum_{i=1}^k\{-nt^{3i+1}+(2n+1)t^{3i}-(n+1)t^{3i-1}\}\\
&-nt+(2n+1)-nt^{-1}\\
&+\sum_{i=1}^k\{-(n+1)t^{-3i+1}+(2n+1)t^{-3i}-nt^{-3i-1}\}.
\end{align*}
Also, the Alexander polynomial of $K_n^{(3,3k+2)}$ is given as 
\begin{align*}
\Delta_{K_n^{(3,3k+2)}}(t)=&\sum_{i=1}^k\{(n+1)t^{3i+1}-(2n+1)t^{3i}+nt^{3i-1}\}\\
&(n+1)t-(2n+1)+(n+1)t^{-1}\\
&+\sum_{i=1}^k\{nt^{-3i+1}-(2n+1)t^{-3i}+(n+1)t^{-3i-1}\}.
\end{align*}
\end{lemma}

%The Alexander polynomial is used after this.

\begin{lemma}\label{Kis(1,1)}
For $n\ge 0$,
the knot $K_n^{(3,q)}$ is a $(1,1)$--knot. In particular, it is a prime knot.
\end{lemma}

\begin{proof}
In Sections \ref{(1,1)diagram} and \ref{calCFK3k+2}, we will give a $(1,1)$--decomposition of $K_n^{(3,q)}$. Also, it is well known that a $(1,1)$--knot is a prime knot (for example, Proposition 4.1 of \cite{Mo}).
\end{proof}

The fact that
$K_n^{(3,q)}$ is a $(1,1)$--knot is useful to calculate $\CFK^{\infty}(K_n^{(3,q)})$.

A knot is called an {\it $L$--space knot\/} if it admits a positive integer Dehn surgery yielding an $L$--space. 
Any positive torus knot is an $L$--space knot.
For details, see \cite{OSlens}.

\begin{lemma}\label{nonLspaceknot}
For $n\ge 1$, the knot $K_n^{(3,q)}$ is not an $L$--space knot. 
\end{lemma}

\begin{proof}
Since non-zero coefficients of the Alexander polynomial of an $L$--space knot are $\pm 1$ \cite{OSlens}, the knot $K_n^{(3,q)}\ (n\ge 1)$ is not an $L$--space knot by Lemma \ref{AlexanderPoly}. (In fact, the knot Floer full complex of an $L$--space knot is a staircase type \cite{Kr}, but that of $K_n^{(3,q)}\ (n\ge 1)$ is not a such type.)
\end{proof}

\begin{lemma}\label{hyperbolic}
For $n\ge 1$, $K_n^{(3,q)}$ is a hyperbolic knot.
\end{lemma}

\begin{proof}
By Lemma \ref{nonLspaceknot}, 
$K_n^{(3,q)}$ is not a positive torus knot.
Since the Alexander polynomial does not change under taking mirror image,
$K_n^{(3,q)}$ is not a negative torus knot.
On the other hand, $K_n^{(3,q)}$ is prime (Lemma \ref{Kis(1,1)}) and admits a three-bridge decomposition, so it is not a satellite knot \cite{Sc}. Therefore, $K_n^{(3,q)}$ is a hyperbolic knot. 
\end{proof}

A knot $K$ is called a {\it Floer thin knot\/} if its knot Floer homology $\widehat{\HFK}_d(K;i)$ is supported in a single grading $\delta=i-d$. The class of Floer thin knots includes all alternating knots, more generally all quasi-alternating knots \cite{MO,OSalt}. Also, the full knot Floer complex of a Floer thin knot is completely determined by its $\tau$ invariant and Alexander polynomial \cite{Pe}.

\begin{lemma}\label{nonFloerthinknot}
The knot $K_n^{(3,q)}$ is not a Floer thin knot.
\end{lemma}

\begin{proof}
By Lemma \ref{hatversion}, 
${\rm rank}\ \widehat{HFK}_{-2j}(K_n^{(3,3k+1)};3k-3j)=2n+1\ne0\ (j=0,\ldots,k)$.
However, $\delta=(3k-3j)-(-2j)=3k-j$ depends on $j$. 
For $K_n^{(3,3k+2)}$, the argument is similar. 
\end{proof}

%%%%%%%%%%%%%%%%%%%%%%%%%%%%%%%%%%%%%%%%%%%%%%%%%%%%%%%%%%%%%%%%%%%%%%%%%%%%%%%%%%%%%%%%%%%%%%%%%%%%%%%%%%%%%%%%
\section{Concordance and Proof of Theorem \ref{mainthm2}}\label{ProofThm2}

In this section, we discuss the concordance relation among our knots. 

When $q\ne q'$, $K_n^{(3,q)}$ and $K_m^{(3,q')}$ are not concordant. This can be seen as follows. Since the concordance invariant $\tau$ is invariant on a stably equivalence class, $\tau(K_n^{(3,q)})=\tau(T(3,q))$ by Lemma \ref{stablyequivalent}. Also, Corollary 1.7 of \cite{OStau} implies $\tau(T(3,q))=q-1$. Thus, $\tau(K_n^{(3,q)})\ne \tau(K_m^{(3,q')})$.

In the following, we consider the case $q=q'$.

\begin{lemma}\label{determinant}
The determinant of $K_n^{(3,q)}$ is given as 
\[
\det(K_n^{(3,q)})=
\begin{cases}
4n+3 & \text{if $q=3k+1$ and $k$ is odd},\\
4n+1 & \text{if $q=3k+1$ and $k$ is even},\\
4n+1 & \text{if $q=3k+2$ and $k$ is odd},\\
4n+3 & \text{if $q=3k+2$ and $k$ is even}.\\
\end{cases}
\]
\end{lemma}
\begin{proof}
It follows by calculating $|\Delta_{K_n^{(3,q)}}(-1)|$ from Lemma \ref{AlexanderPoly}.
\end{proof}

\begin{lemma}\label{nonconcordant}
$K_n^{(3,q)}$ is not concordant to $K_m^{(3,q)}$ if either of the following is satisfied\textup{:}
\begin{itemize}
\item $q=3k+1$, $k$ is odd and $(4n+3)(4m+3)$ is not a square integer;
\item $q=3k+1$, $k$ is even and $(4n+1)(4m+1)$ is not a square integer;
\item $q=3k+2$, $k$ is odd and $(4n+1)(4m+1)$ is not a square integer; or
\item $q=3k+2$, $k$ is even and $(4n+3)(4m+3)$ is not a square integer.
\end{itemize}
\end{lemma}
\begin{proof}
If $\det(K_n^{(3,q)}\# -K_m^{(3,q)})=\det(K_n^{(3,q)})\det(K_m^{(3,q)})$ is not a square integer, then $K_n^{(3,q)}\# -K_m^{(3,q)}$ is not a slice knot, that is, $K_n^{(3,q)}$ is not concordant to $K_m^{(3,q)}$ \cite{FM}. Thus, Lemma \ref{determinant} implies the conclusion.
\end{proof}

\begin{proof}[Proof of Theorem \ref{mainthm2}]
Suppose that $q=3k+1$ and $k$ is odd. If $n$ and $m$ are distinct non-negative integers such that $4n+3$ and $4m+3$ are prime, then $K_n^{(3,q)}$ is not concordant to $K_m^{(3,q)}$ by Lemma \ref{nonconcordant}. Since there are infinitely many prime integers of form $4n+3$, the claim is true when $q=3k+1$ and $k$ is odd.
The other cases can be shown in the same manner.
 \end{proof}
 
 \begin{remark}
We expect that $K_n^{(3,q)}$ and $K_m^{(3,q)}$ are not concordant for $n\ne m$, but we could not prove it. For several $n,m$, we can verify that they are not concordant by the irreducibility of the Alexander polynomial or the Levine--Tristram signature.  
 \end{remark}
%%%%%%%%%%%%%%%%%%%%%%%%%%%%%%%%%%%%%%%%%%%%%%%%%%%%%%%%%%%%%%%%%%%%%%%%%%%%%%%%%%%%%%%%%%%%%%%%%%%%%%%%%%%%%%%%
\section{$(1,1)$--diagram of the knot $K_n^{(3,3k+1)}$}\label{(1,1)diagram}
In this section, we give a $(1,1)$--diagram of $K_n^{(3,3k+1)}$, following the process described in \cite{GMM,Va}. 

\begin{definition}
Let $K$ be a knot in $S^3$, $\Sigma$ be the torus, $\alpha,\beta$ be simple closed curves in $\Sigma$ and $z,w$ be two points in $\Sigma$ disjoint from $\alpha,\beta$. A tuple $(\Sigma;\alpha,\beta;z,w)$ is called a {\it $(1,1)$--diagram\/} of $K$ if
\begin{itemize}
\item $(\Sigma;\alpha,\beta)$ is the standard genus one Heegaard diagram of $S^3$,
\item $K=t\cup t'$, where $t$ (resp. $t'$) is a trivial arc connecting $z,w\in\Sigma$ in the outer (resp. inner) solid torus such that it is disjoint from the meridian disk $D_{\beta}$ (resp. $D_{\alpha}$) bounded by $\beta$ (resp. $\alpha$).
\end{itemize}
\end{definition}

We start from the standard genus one Heegaard splitting $(V_{\alpha},V_{\beta})$ of $S^3$, which has meridian curves $\alpha$ and $\beta$ respectively. We put a boundary parallel arc $t$ in the outer solid torus $V_{\beta}$ such that it is disjoint from the meridian disk bounded by $\beta$. For the endpoints of $t$, the white dot is $z$ and the black dot is $w$. To make a $(1,1)$--diagram, move $t$ so that the union of $t$ with $t'$ forms the knot $K_n^{(3,3k+1)}$, where $t'$ is another boundary parallel arc connecting $z,w$ in the inner solid torus $V_{\alpha}$.

In Figure \ref{1,1diagram-K_n3,3k+1-1-4}, we move the arc $t$ so that the union of $t$ with $t'$ forms the knot $K_n^{(3,3k+1)}$. The $\beta$ curve (blue) is moved so as not to touch two endpoints of $t$. We merge parallel arcs into one curve with multiplicity, see Figure \ref{merging}. To complete the procedure, we need to add $n$-twists in a clockwise direction in the region indicated the dotted box. 

\begin{figure}[h]
\begin{tabular}{ccc}
\begin{minipage}{.50\textwidth}
\centering
\includegraphics[scale=0.25]{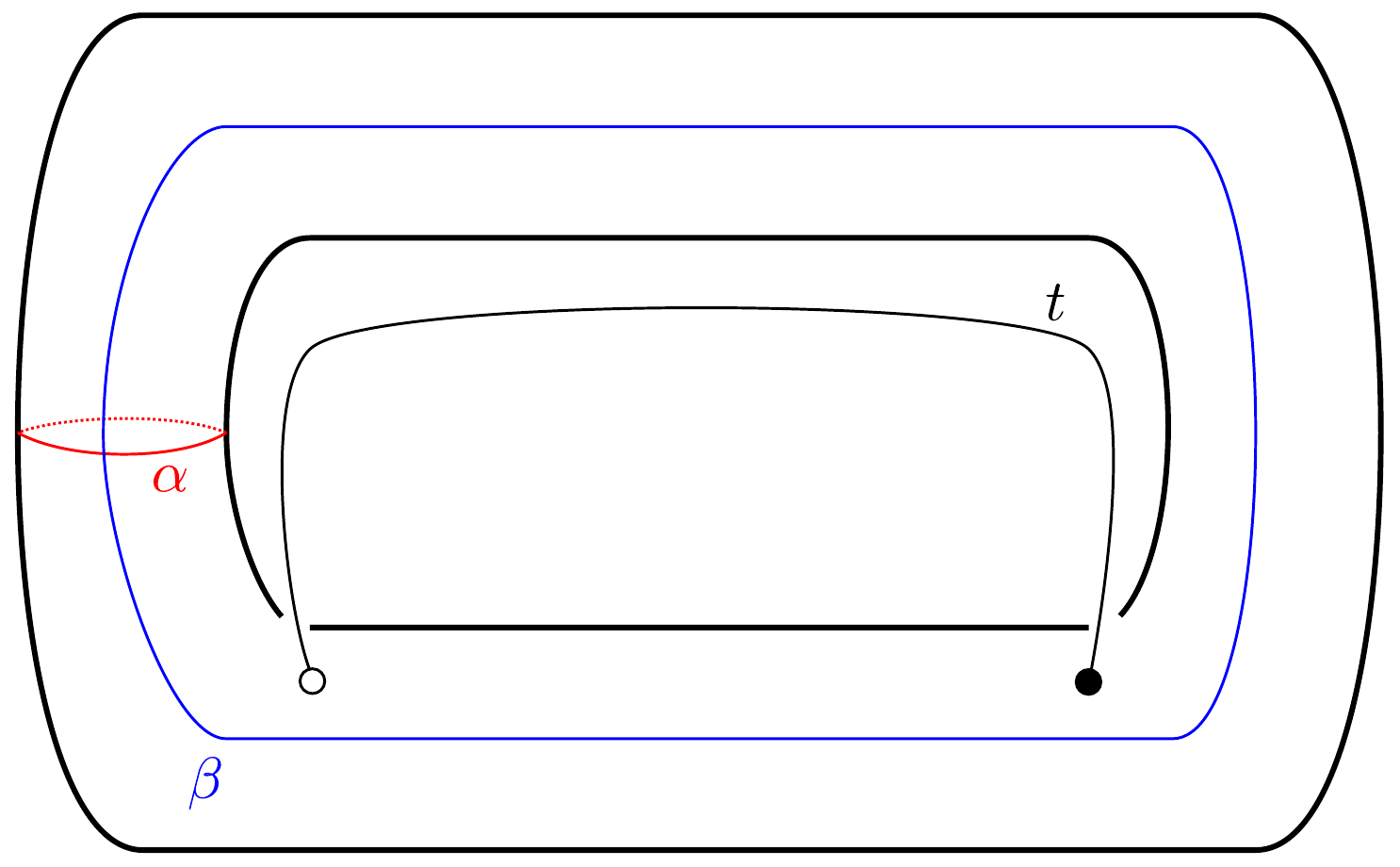}
\end{minipage}
$\longrightarrow$
\begin{minipage}{.50\textwidth}
\centering
\includegraphics[scale=0.25]{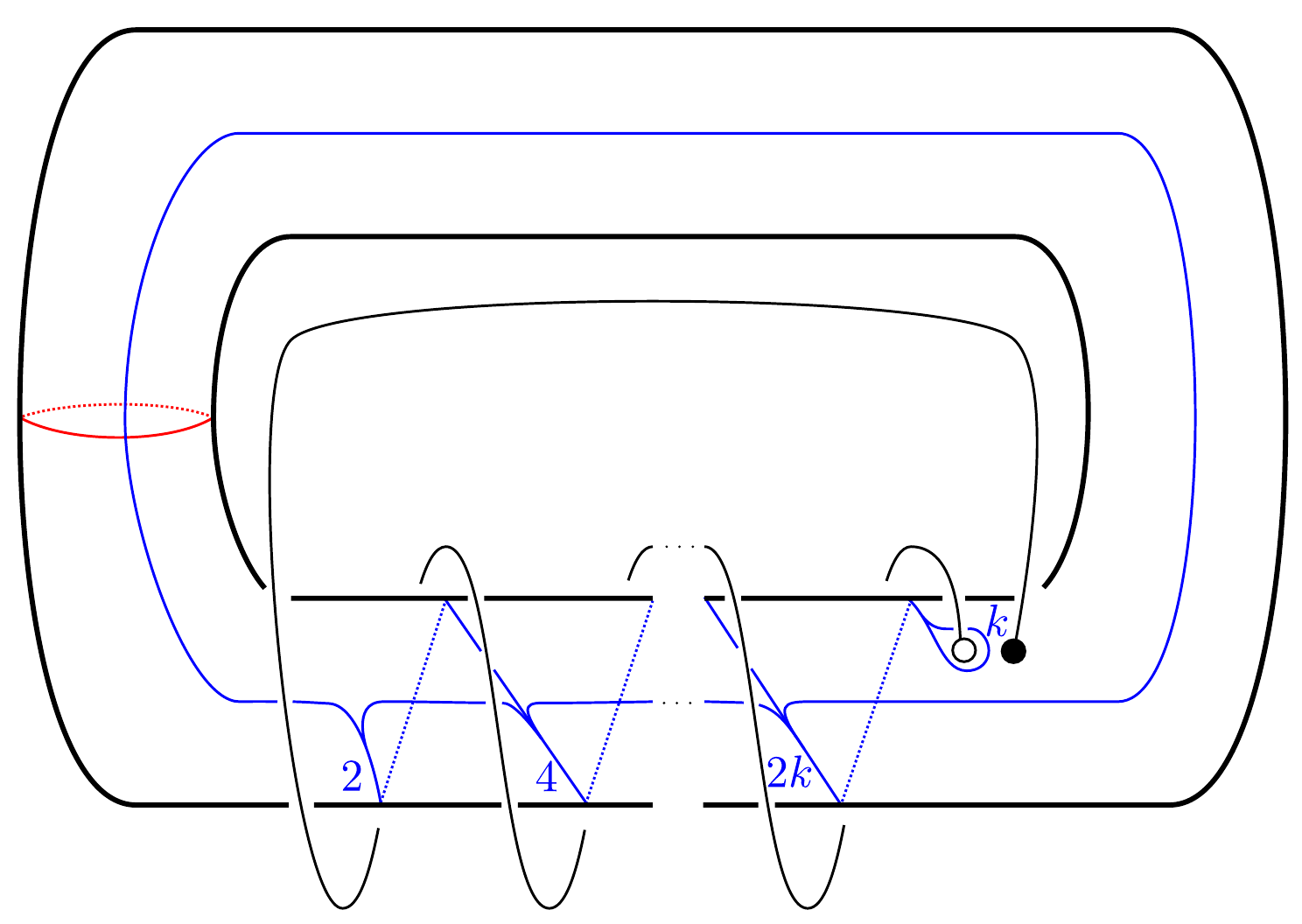}
\end{minipage}\\
$\swarrow$\\
\begin{minipage}{.50\textwidth}
\centering
\includegraphics[scale=0.25]{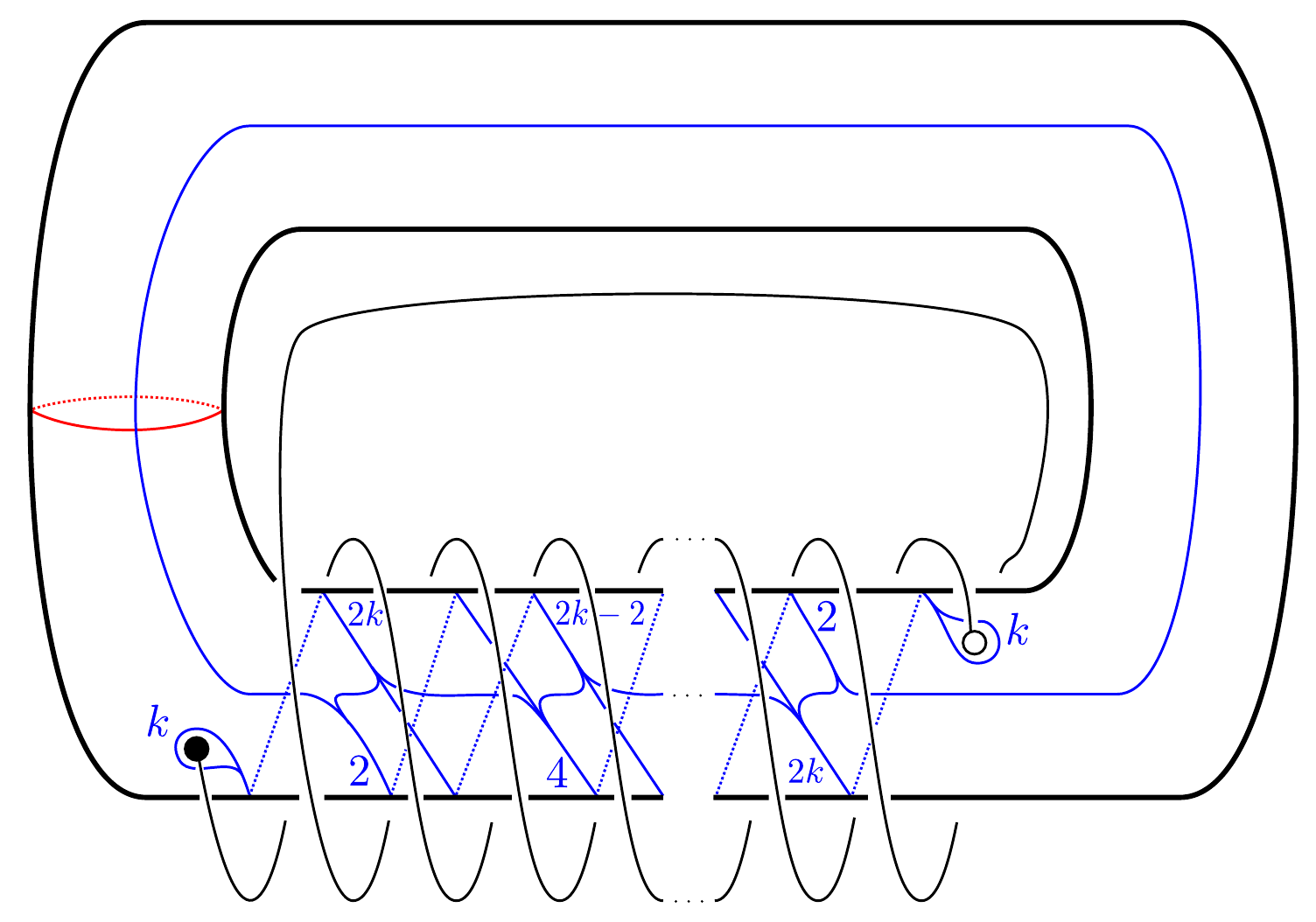}
\end{minipage}
$\longrightarrow$
\begin{minipage}{.50\textwidth}
\centering
\includegraphics[scale=0.25]{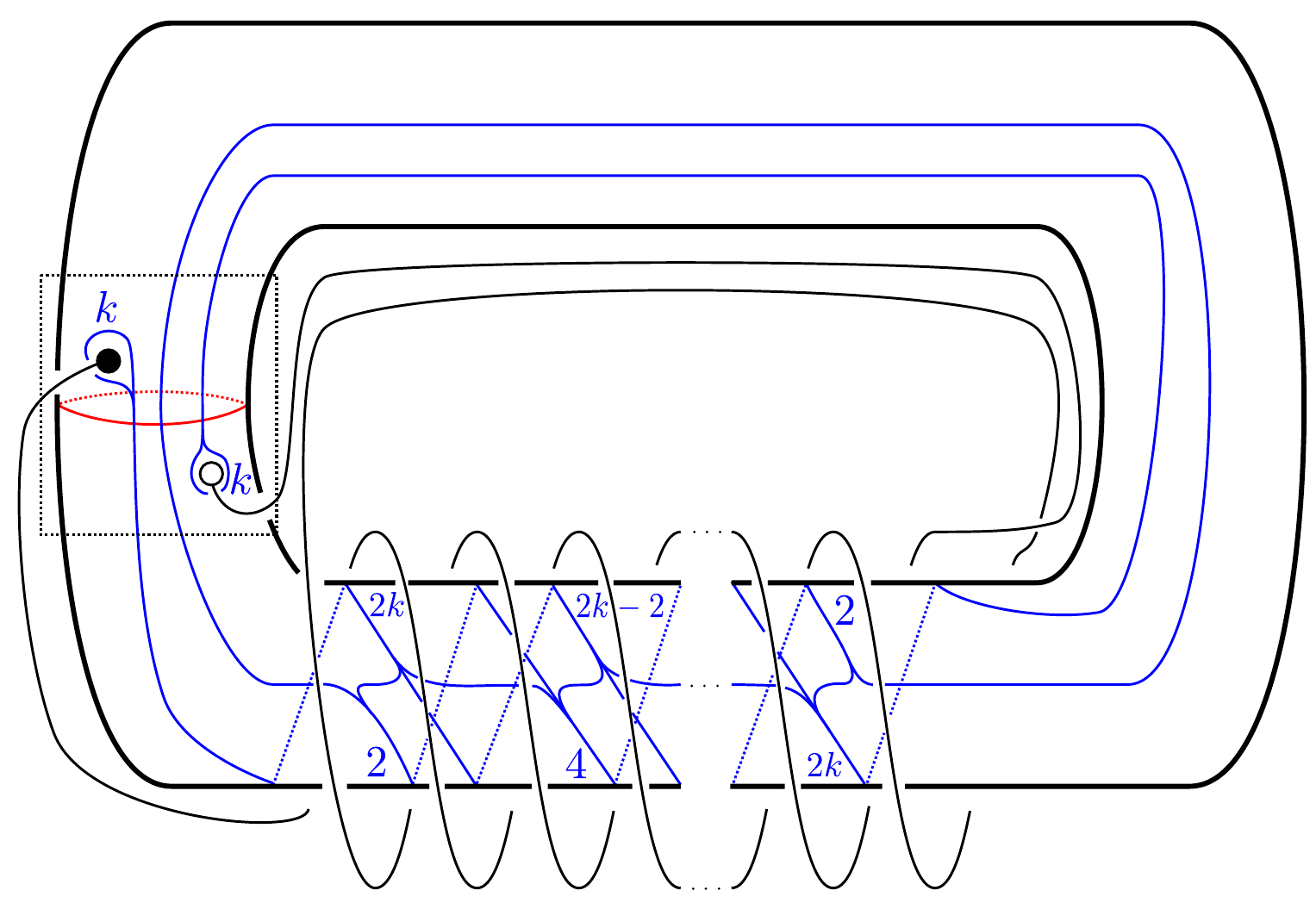}
\end{minipage}\\
\end{tabular}
\caption{How to move the arc $t$ to get a $(1,1)$--diagram of $K_n^{(3,3k+1)}$. The bottom right is a $(1,1)$--diagram of $K_0^{(3,3k+1)}=T(3,3k+1)$. To obtain a $(1,1)$--diagram of $K_n^{(3,3k+1)}$, we need to add $n$-twists clockwisely in the region indicated the dotted box.}
\label{1,1diagram-K_n3,3k+1-1-4}
\end{figure}

\begin{figure}[h]
\centering
\includegraphics[scale=0.3]{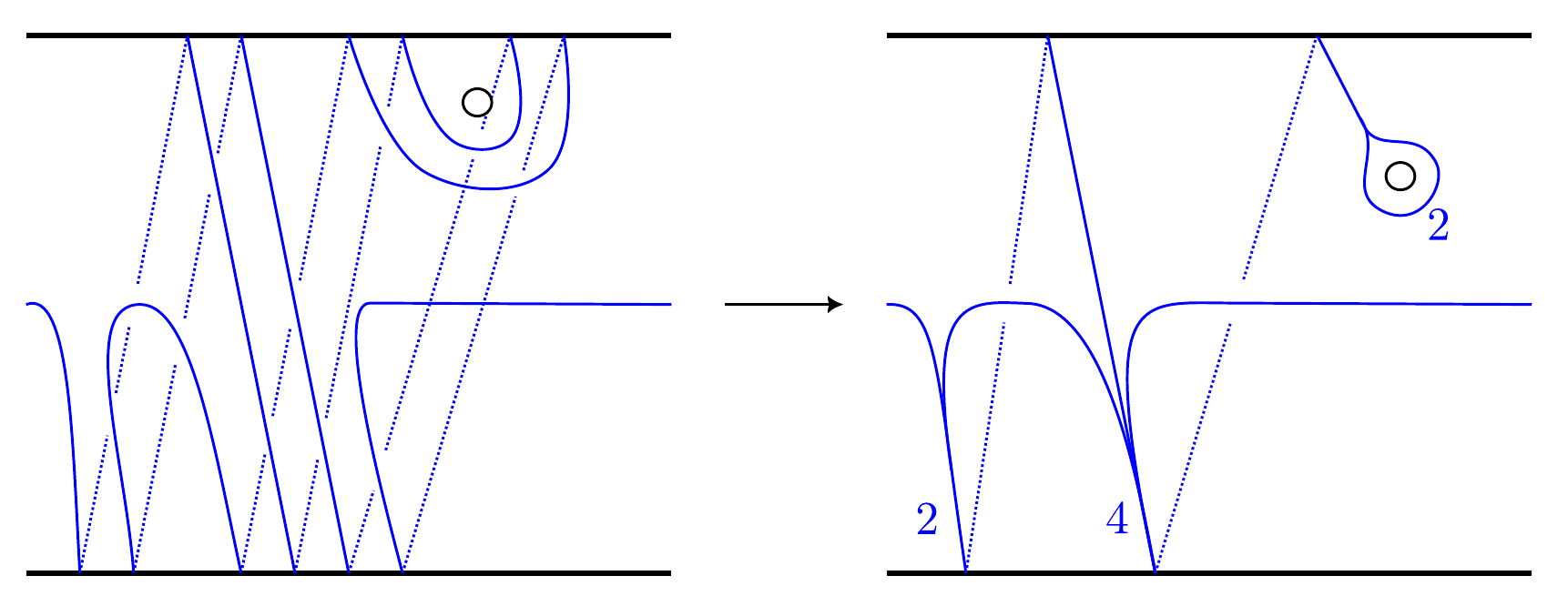}
\caption{Parallel arcs are merged into one curve with multiplicity.}
\label{merging}
\end{figure}

Figure \ref{1,1diagram-K_n3,3k+1-5-6} shows that the $\beta$ curve is neatly aligned in a complete $(1,1)$-diagram of $K_0^{(3,3k+1)}$ drawn without the arc $t$.

\begin{figure}[h]
\begin{tabular}{c}
\begin{minipage}{.50\textwidth}\centering
\includegraphics[scale=0.25]{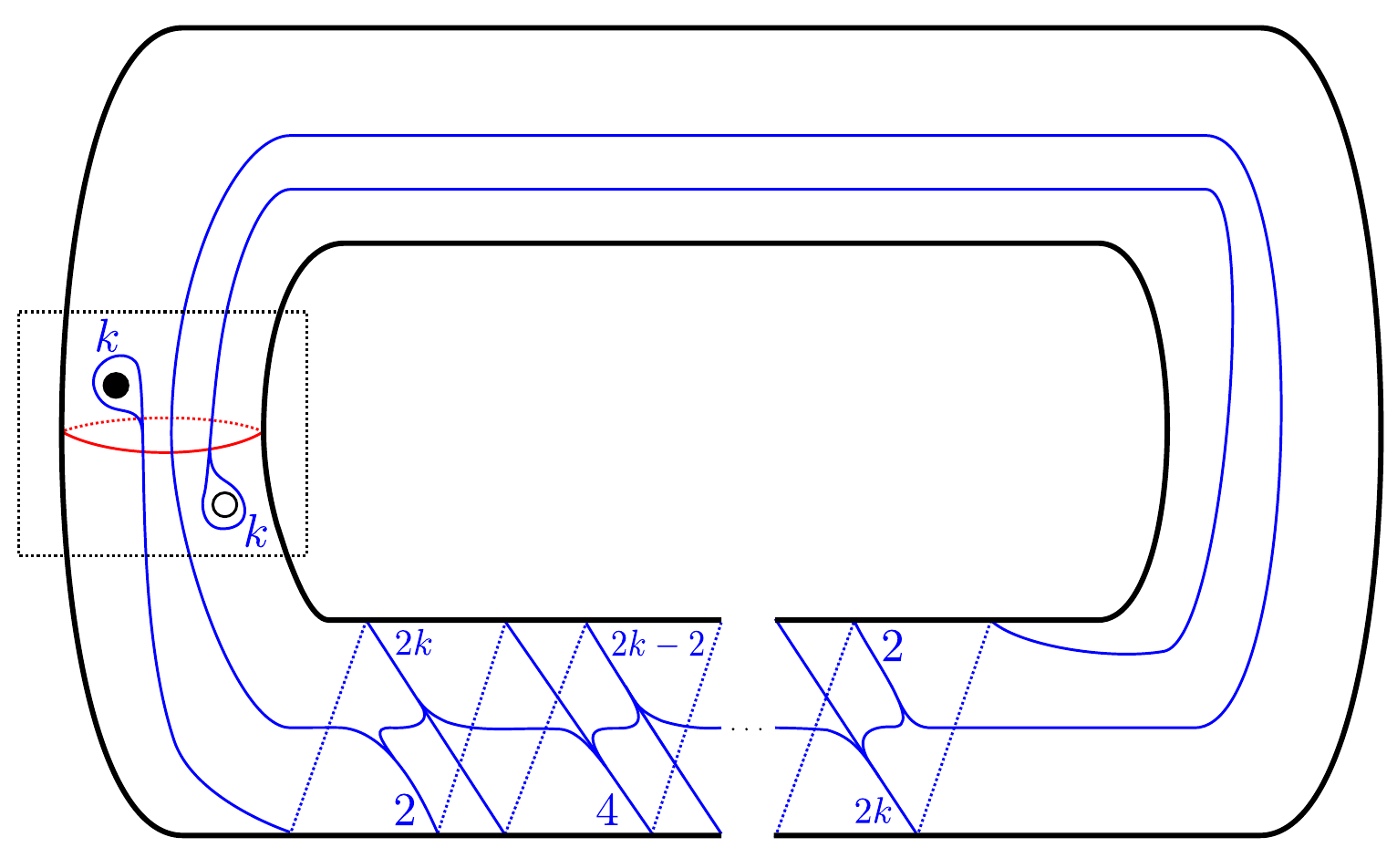}
\end{minipage}
$\longrightarrow$
\begin{minipage}{.50\textwidth}\centering
\includegraphics[scale=0.25]{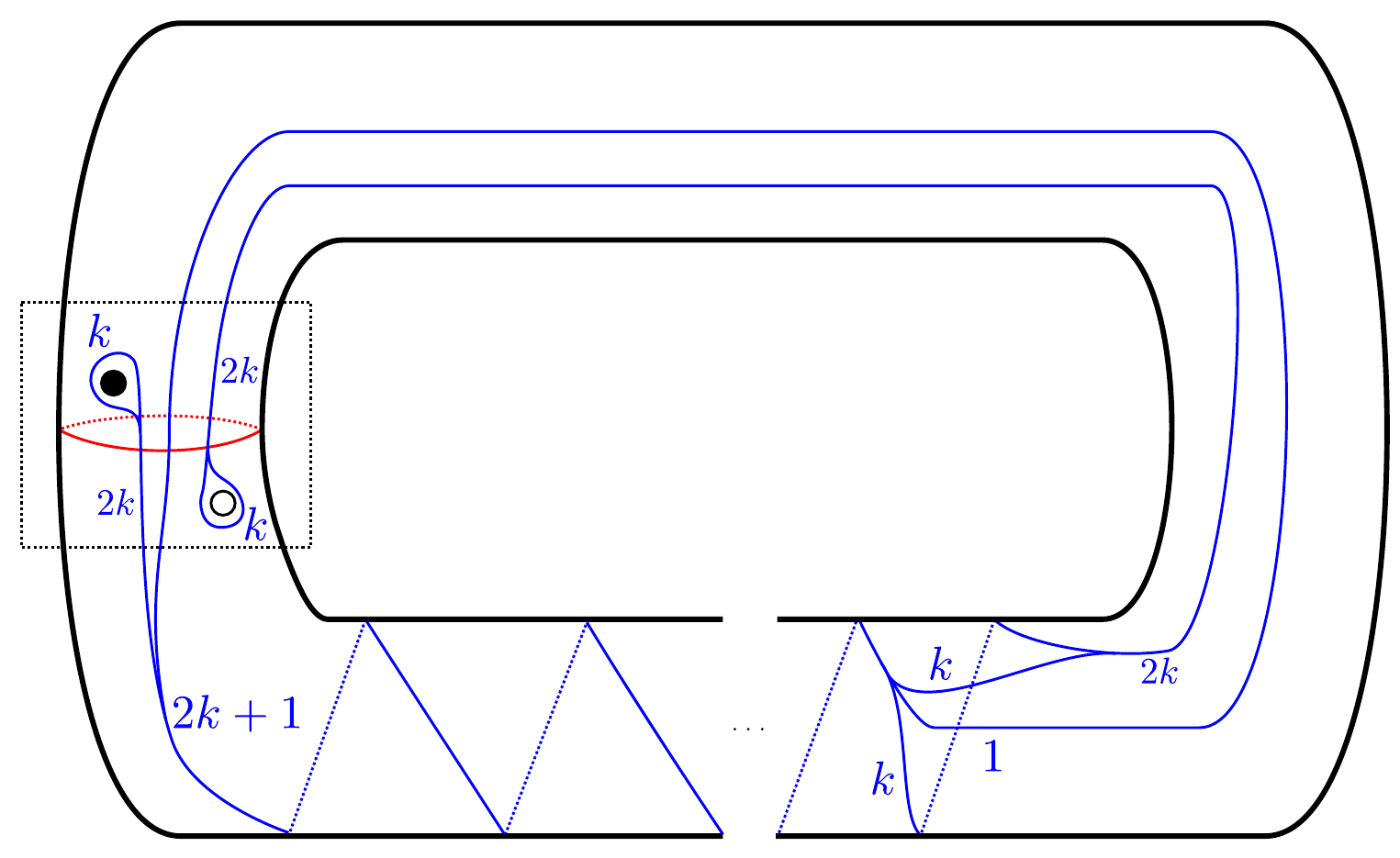}
\end{minipage}\\
\includegraphics[scale=0.4]{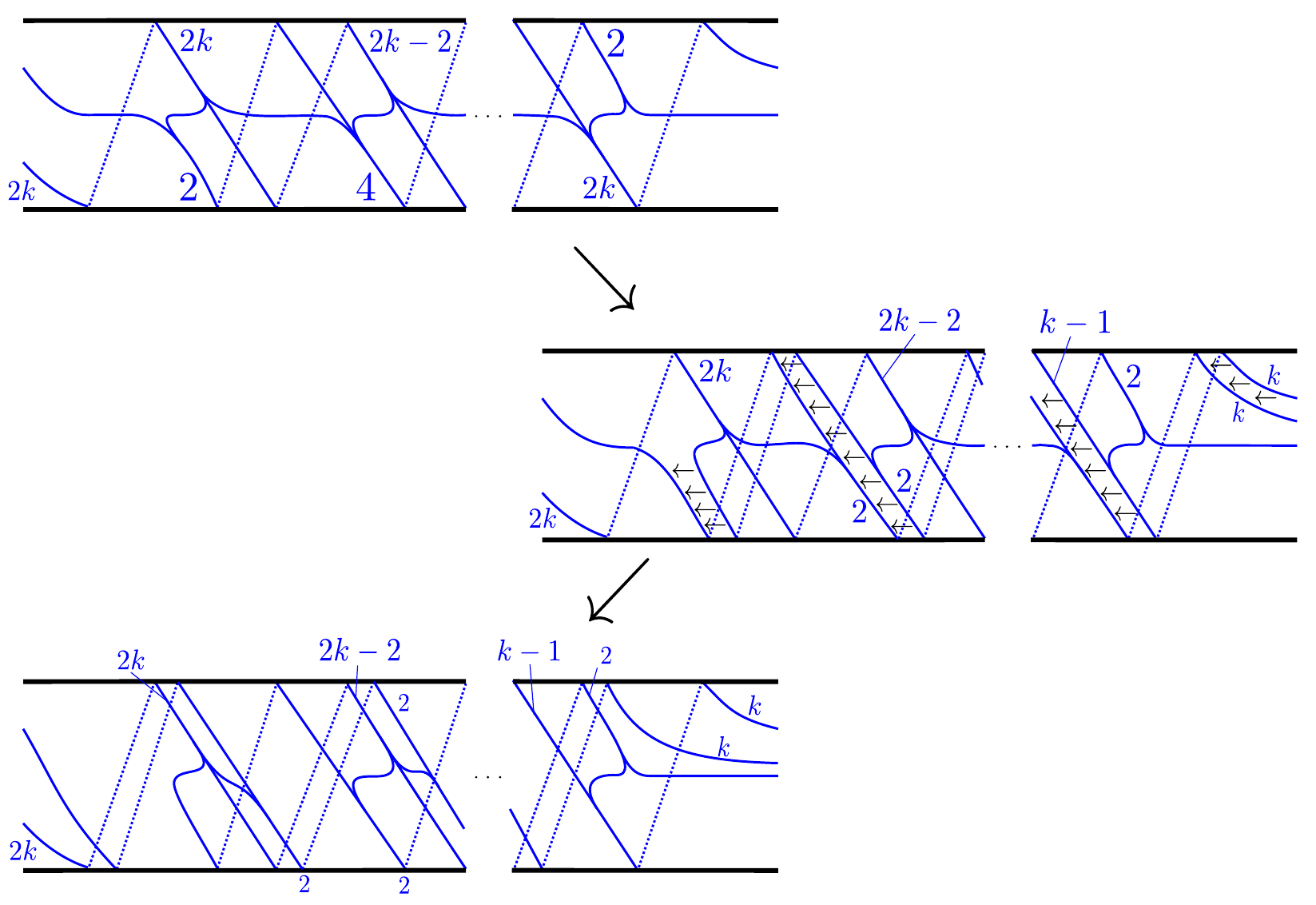}
\end{tabular}
\caption{The $\beta$ curve is neatly aligned. Bottom is a detailed description of how to aline the $\beta$ curve.}
\label{1,1diagram-K_n3,3k+1-5-6}
\end{figure}

To simplify the calculation of $\CFK^{\infty}$, we perform two further operations. First, in Figure \ref{1,1diagram-K_n3,3k+1-afterDhentwists}, we apply left handed Dehn twists along the $\alpha$ curve to the $(1,1)$-diagram of $K_0^{(3,3k+1)}$. These twists do not affect the calculation of $\CFK^{\infty}$. Second, it is convenient to consider the universal cover of a $(1,1)$--diagram. We cut the torus along $\alpha$ and the standard longitude curve, then the result is shown in Figure \ref{1,1rectangle-K_n3,3k+1} where the rectangle restores the torus by identifying the left side with the right side, and the top and bottom as usual.  

\begin{figure}[h]
\centering
\includegraphics[scale=0.3]{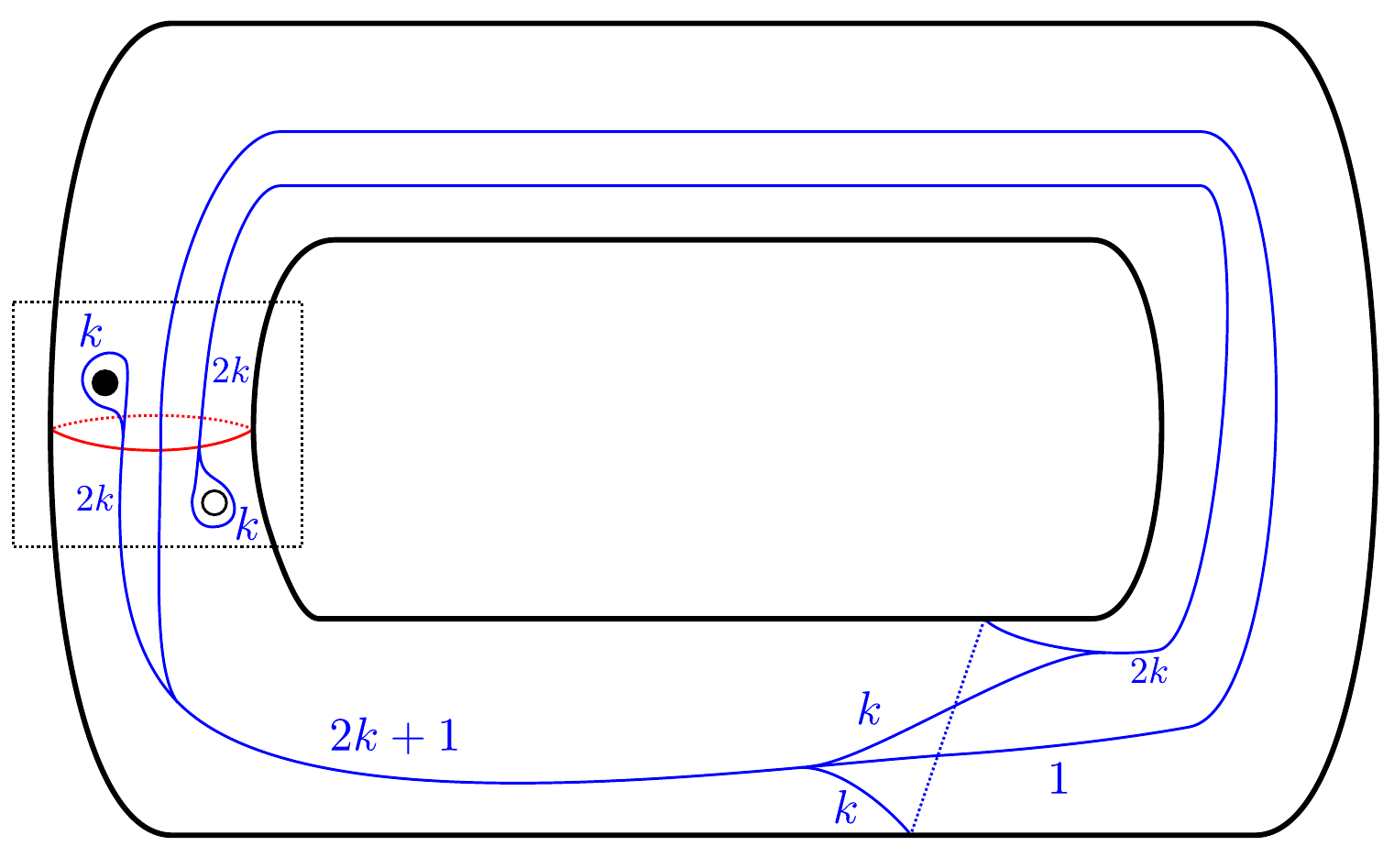}
\caption{A $(1,1)$-diagram of $K_0^{(3,3k+1)}$ after left handed Dehn twists along the $\alpha$ curve.}
\label{1,1diagram-K_n3,3k+1-afterDhentwists}
\end{figure}

\begin{figure}[h]
\centering
\includegraphics[scale=0.3]{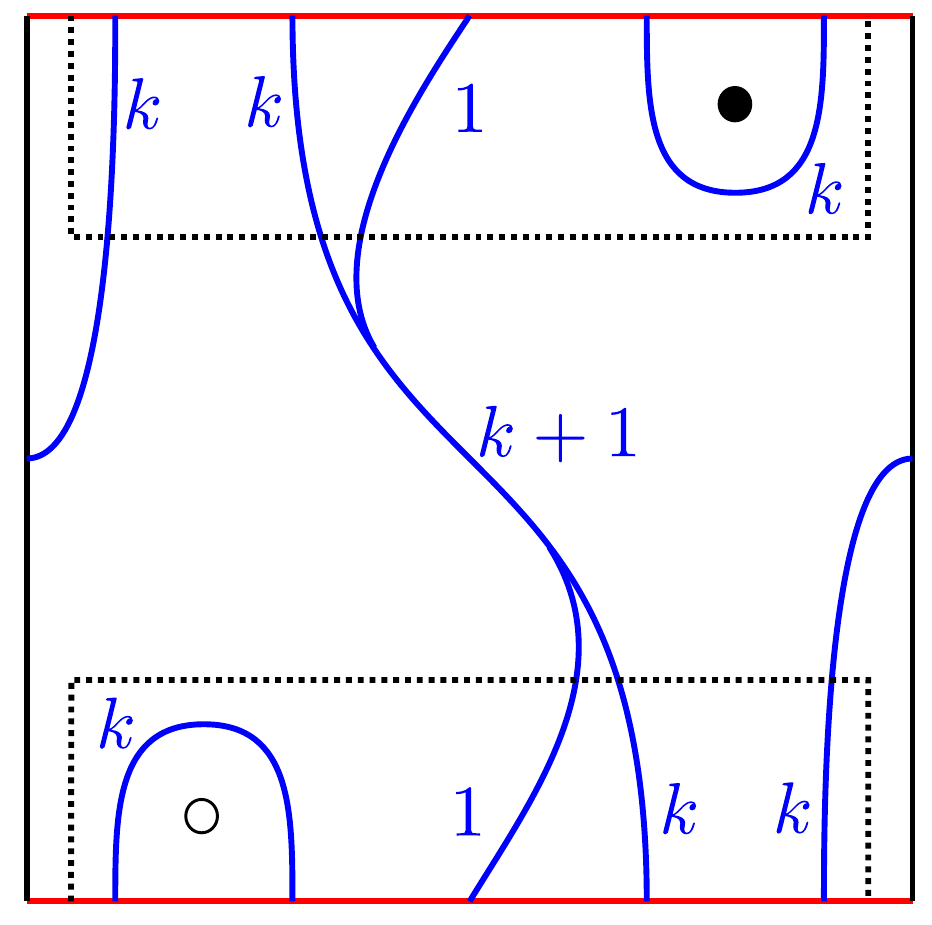}
\caption{The rectangle is obtained by cutting the diagram in Figure \ref{1,1diagram-K_n3,3k+1-afterDhentwists} along $\alpha$ and the standard longitude. When it restores the torus, the top and bottom sides (red) become the $\alpha$ curve.}
\label{1,1rectangle-K_n3,3k+1}
\end{figure}

%%%%%%%%%%%%%%%%

%%%%%%%%%%%%%%%%%%%%%%%%%%%%%%%%%%%%%%%%%%%%%%%%%%%%%%%%%%%%%%%%%%%%%%%%%%%%%%%%%%%%%%%%%%%%%%%%%%%%%%%%%%%%%%%%%%

\section{Calculation of  the full knot Floer complex $\CFK^{\infty}(K_n^{(3,3k+1)})$}\label{calCFK3k+1}

We can combinatorially calculate the full knot Floer complex $\CFK^{\infty}(K_n^{(3,q)})$ from its $(1,1)$--diagram (see, for example, \cite{GMM,OS,Va}).
We now explain how to get the full knot Floer complex of a $(1,1)$--knot $K$. Let $(\Sigma;\alpha,\beta;z,w)$ be its $(1,1)$--diagram. 

First, for $x\in\alpha\cap\beta$, the element $[x,i,j]$ with the Alexander grading $A(x)=j-i$ becomes a generator of $\CFK^{\infty}(K)$. 

Second, to define the differentials, we consider a {\it Whitney disk\/} which is a bigon formed by an arc of $\alpha$ and an arc of $\beta$. 
More precisely, a Whitney disk $\phi\colon D\to\Sigma$ (where $D\subset\C$ is the closed unit disk) maps the right half-arc to an arc of $\alpha$ and the left half-arc to an arc of $\beta$, see Figure \ref{WhitneyDisk}. Then, it is called a {\it Whitney disk connecting $\phi(-i)$ to $\phi(i)$}. Moreover, for a Whitney disk $\phi$, let $n_z(\phi)$ (resp. $n_w(\phi)$) be the algebraic intersection number between $\phi$ and $z$ (resp. $w$) in $\Sigma$. With these settings, we define $\partial [x,i,j]$ as 
\[
\partial [x,i,j] =\sum_{y\in\alpha\cap\beta}\sum_{\phi}[y,i-n_w(\phi),j-n_z(\phi)]
\]
where $\phi$ runs over Whitney disks connecting $x$ to $y$, which hold $n_z(\phi), n_w(\phi)\ge 0$ and have the Maslov index one. It is hard to give the definition of the Maslov index here, but the formula of \cite{Li} is helpful (also, Lemma 6.5 of \cite{OS} is useful). In the following, we consider only Whitney disks with $n_z(\phi), n_w(\phi)\ge 0$ and Maslov index one. Also, it is convenient to use the universal cover of the torus $\Sigma$ to find such disks. Moreover, it suffices to consider one connected component $\tilde{\alpha}$ (resp. $\tilde{\beta}$) of lifts of $\alpha$ (resp. $\beta$). 

\begin{figure}[h]
\centering
\includegraphics[scale=0.3]{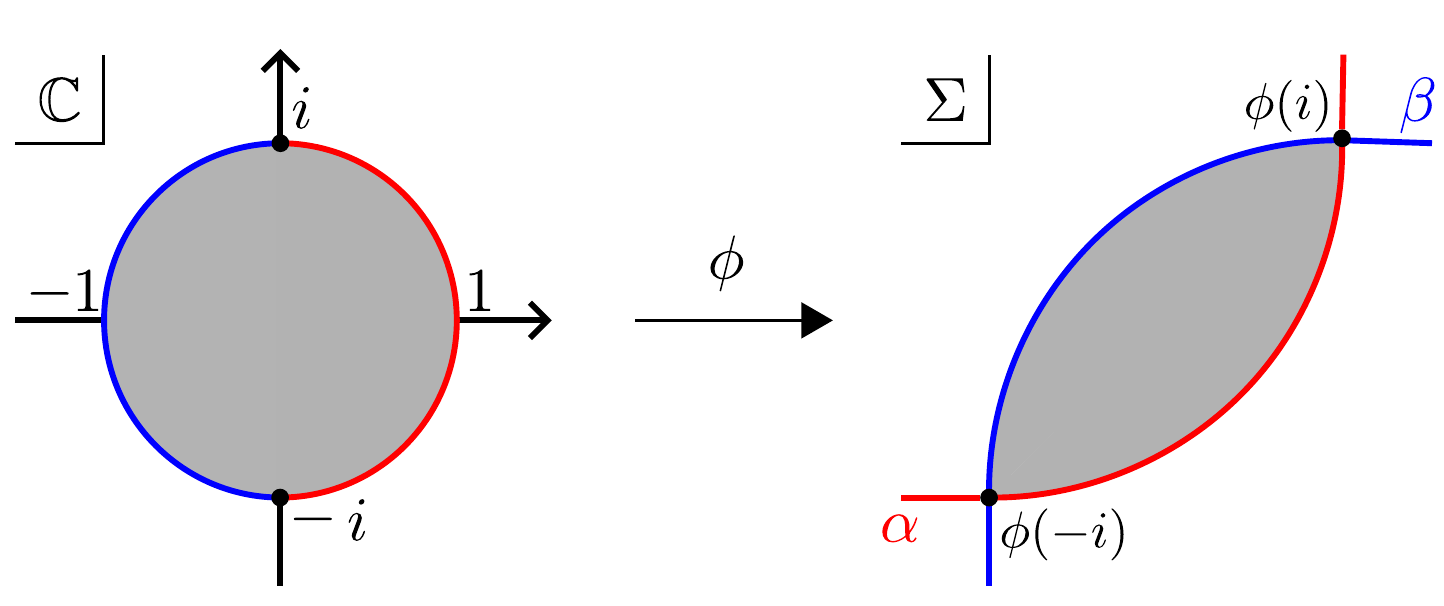}
\caption{A Whitney disk connecting $\phi(-i)$ to $\phi(i)$.}
\label{WhitneyDisk}
\end{figure}

Third, the Alexander grading $A\colon\alpha\cap\beta\to\Z$ is uniquely determined by following two properties:
\begin{itemize}
\item for a Whitney disk $\phi$ connecting $x$ to $y$, 
\[
A(x)-A(y)=n_z(\phi)-n_w(\phi)
\]
holds; and
\item symmetric, that is, $\#\{x\mid A(x)=i\}=\#\{x\mid A(x)=-i\}$ for any $i\in\Z$.
\end{itemize}

Finally, the Maslov grading $M\colon \alpha\cap\beta\to\Z$ can be given as follows. We consider only generators of form $[x,0,j]$ and differentials derived from a Whitney disk with $n_w=0$. Then, its homology represents $\widehat{HF}(S^3)\cong\F_2$, and the Maslov grading of its generator is set to be zero. This fixes the Maslov grading of the other generators of $\CFK^{\infty}(K)$ by the differentials and the action of $U$.

In this section, we demonstrate calculations of $\CFK^{\infty}(K_1^{(3,4)})$ and $\CFK^{\infty}(K_2^{(3,7)})$. Thereafter, the general case $q=3k+1$ is considered.

%%%%%%%%%%%%%%%%%%%%%%%%%%%%%%%%%%%%%%%%%%%%%%%%%%
\subsection{$\CFK^{\infty}(K_1^{(3,4)})$ (the case $q=3\cdot 1+1,\ n=1$)}
Figure \ref{1,1universal-K_1+3,4} shows the universal cover of a $(1,1)$--diagram of $K_1^{(3,4)}$. The intersection points of $\tilde{\alpha}$ and $\tilde{\beta}$ are labeled by 
\[
b^1_1,b^1_2,b^1_3,a^1_3,a^1_2,a^1_1,f_1,f_2,g,e_2,e_1,c^1_1,c^1_2,c^1_3,d^1_3,d^1_2,d^1_1
\]
along $\tilde{\beta}$ from left-upper to right-lower. 

\begin{figure}[h]
\centering
\begin{tabular}{c}
\includegraphics[scale=0.4]{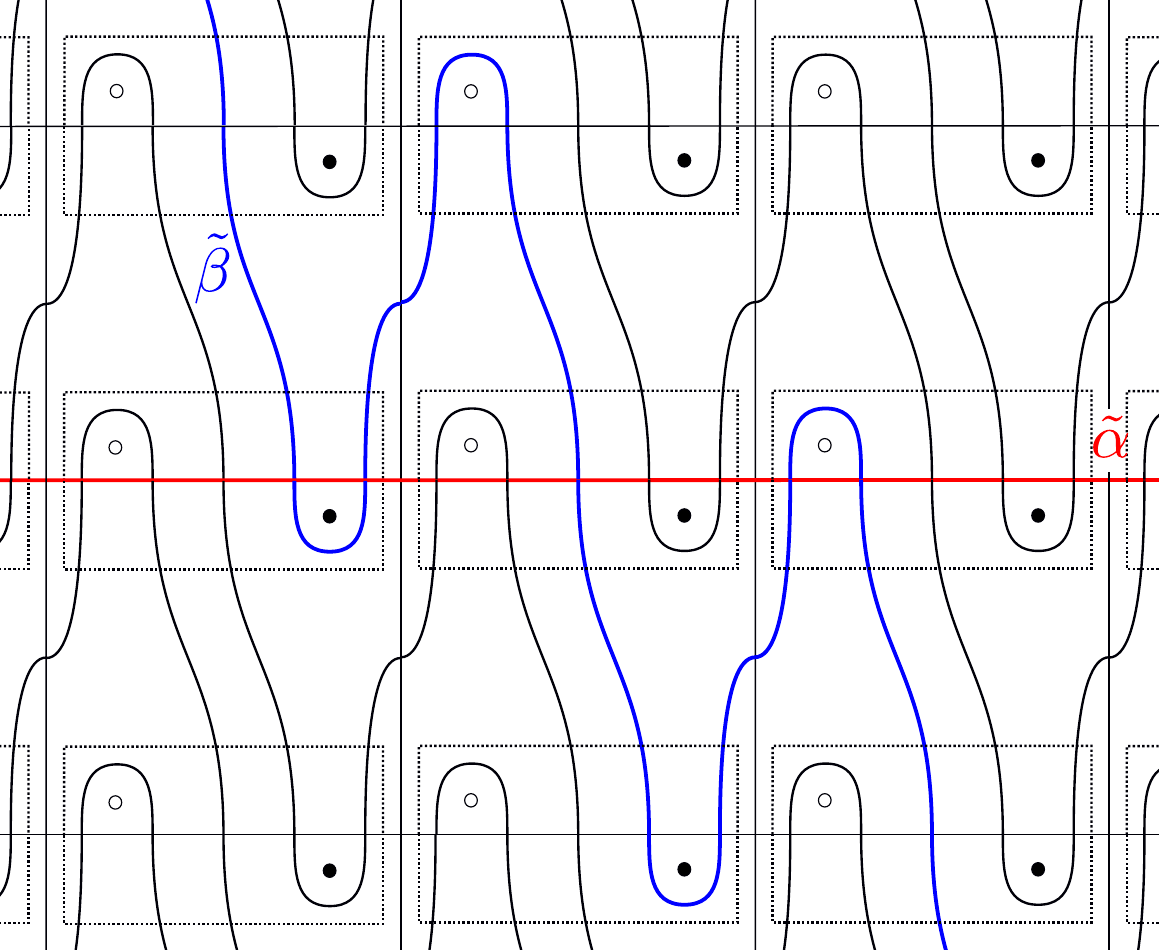}\\
$\downarrow$\\
\includegraphics[scale=0.37]{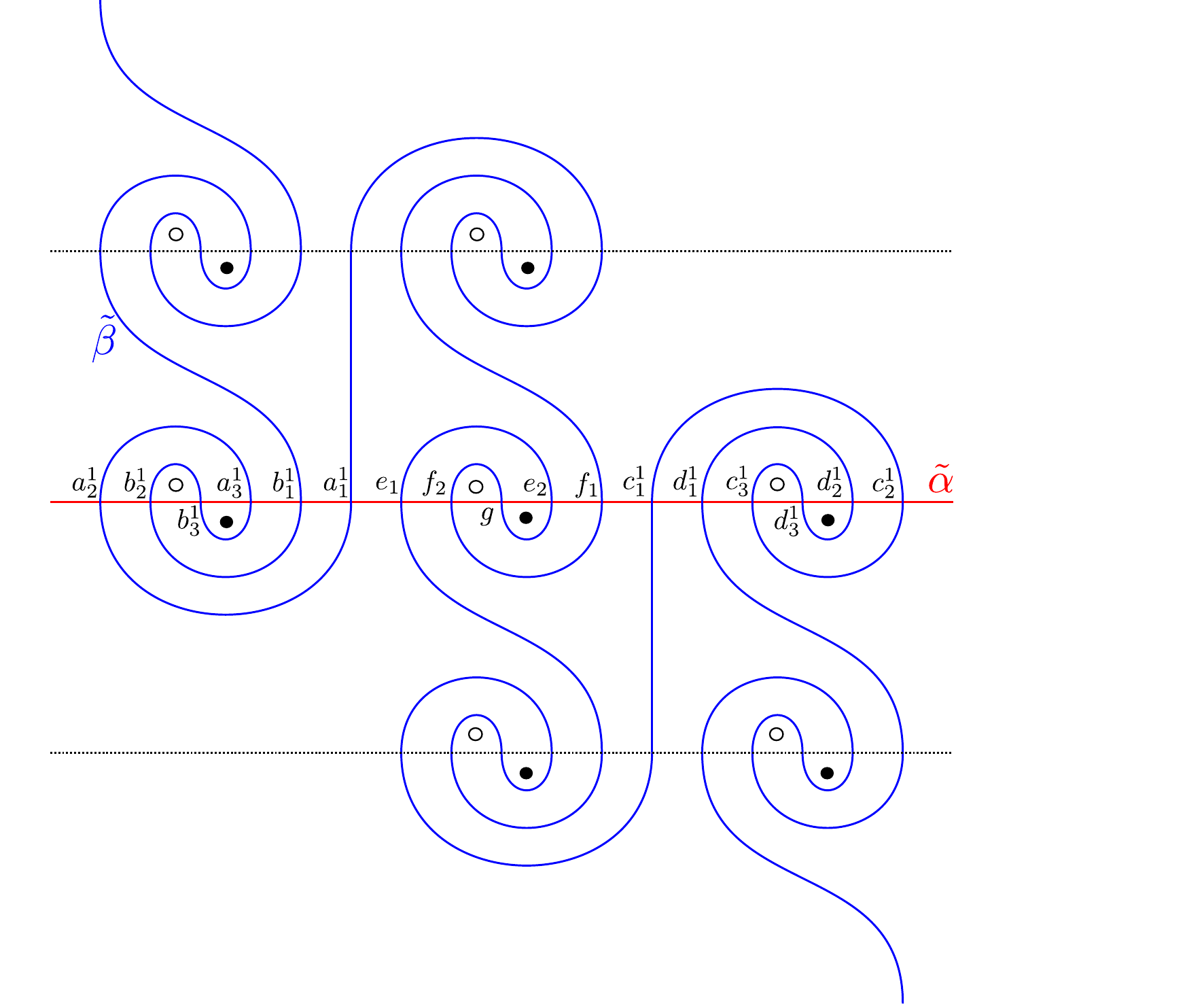}
\end{tabular}
\caption{(Top) The universal cover of a $(1,1)$--diagram of $K_0^{(3,4)}$ (see Figure \ref{1,1rectangle-K_n3,3k+1}). Let $\tilde{\alpha}$ (red) and $\tilde{\beta}$ (blue) be one connected components of lifts of $\alpha$ and $\beta$, respectively. (Bottom) The universal cover of a $(1,1)$--diagram of $K_1^{(3,4)}$. It can be obtained from the top figure by  twisting  two points once  clockwisely in the dotted box. The labeling of $\tilde{\alpha}\cap\tilde{\beta}$ is indicated to the left of each point. Recall that the white dot is $z$ and the black dot is $w$.}
\label{1,1universal-K_1+3,4}
\end{figure}

Then, Whitney disks contributing to the differentials of $\CFK^{\infty}(K_1^{(3,4)})$ can be seen as in Table \ref{differential-K_1^{3,4}}. For example, there is a Whitney disk connecting $a^1_1$ to $a^1_2$ with black dot, which is shaped like a semi-disk. On the other hand, there is only one disk connecting $a^1_1$ to $a^1_3$, but it has the Maslov index two, so it does not contribute to the differential. In this way, all Whitney disks are found. Also, a Whitney disk from $b^1_3,\ d^1_3$ or $g$ does not exist. 

\begin{table}[h]
\begingroup
\renewcommand{\arraystretch}{1.2}
\begin{tabular}{|c|c|c||c|c|c|}
\hline
from & to & base point & from & to & base point\\
\hline
\multirow{4}{*}{$a^1_1$} & $a^1_2$ & $\bullet$ & \multirow{4}{*}{$c^1_1$} & $c^1_2$ & $\circ$ \\
& $b^1_1$ & $\bullet$ && $d^1_1$ & $\circ$ \\ 
& $f_1$ & $\circ\circ$ && $f_1$ & $\bullet\bullet$\\
& $e_1$ & $\circ\circ$ && $e_1$ & $\bullet\bullet$\\
\hline
\multirow{2}{*}{$a^1_2$} & $a^1_3$ & $\circ$ & \multirow{2}{*}{$c^1_2$} & $c^1_3$ & $\bullet$\\
& $b^1_2$ & $\bullet$ && $d^1_2$ & $\circ$\\
\hline
$a^1_3$ & $b^1_3$ & $\bullet$ & $c^1_3$ & $d^1_3$ & $\circ$\\
\hline
\multirow{2}{*}{$b^1_1$} & $a^1_3$ & $\circ$ & \multirow{2}{*}{$d^1_1$} & $c^1_3$ & $\bullet$ \\
& $b^1_2$ & $\bullet$ && $d^1_2$ & $\circ$\\
\hline
$b^1_2$ & $b^1_3$ & $\circ$ & $d^1_2$ & $d^1_3$ & $\bullet$\\
\hline
\multirow{2}{*}{$f_1$} & $f_2$ & $\bullet$ & \multirow{2}{*}{$e_1$} & $f_2$ & $\bullet$\\
& $e_2$ & $\circ$ && $e_2$ & $\circ$\\
\hline
$f_2$ & $g$ & $\circ$ & $e_2$ & $g$ & $\bullet$\\
\hline 
\end{tabular}
\endgroup
\caption{The list of Whitney disks contributing to differentials of $\CFK^{\infty}(K_1^{(3,4)})$. For example, there is a Whitney disk connecting $a^1_1$ to $a^1_2$ with one black dot.}
\label{differential-K_1^{3,4}}
\end{table}

The Alexander gradings of generators are given as in Table \ref{Alexander-K_1^{3,4}}. This list can be obtained from Whitney disks in Table \ref{differential-K_1^{3,4}}. For example, a Whitney disk connecting $a^1_1$ to $a^1_2$ with one black dot implies $A(a^1_1)-A(a^1_2)=-1$. In the same way, we can compare the Alexander gradings of all generators. Finally, based on the symmetry of the Alexander grading, we can determine the gradings. 

\begin{table}[h]
\begingroup
\renewcommand{\arraystretch}{1.3}
\begin{tabular}{|c|l|}
\hline
Alexander grading & generator\\
\hline
$4$ & $b^1_2$\\
\hline
$3$ & $a^1_2,\ b^1_1,\ b^1_3$\\
\hline
$2$ & $a^1_1,\ a^1_3$\\
\hline
$1$ & $f_2$\\
\hline
$0$ & $g,\ f_1,\ e_1$\\
\hline
$-1$ & $e_2$\\
\hline
$-2$ & $c^1_1,\ c^1_3$\\
\hline
$-3$ & $c^1_2,\ d^1_1,\ d^1_3$\\
\hline
$-4$ & $d^1_2$\\
\hline
\end{tabular}
\endgroup
\caption{The Alexander gradings of the generators of $\CFK^{\infty}(K_1^{(3,4)})$.}
\label{Alexander-K_1^{3,4}}
\end{table}

The left of Figure \ref{CFK-K_1+3,4} shows the description of $\CFK^{\infty}(K_1^{(3,4)})$. The generators $[f_1,i,i],[e_1,i,i]$ and $[g,i,i]$ lie on the diagonal line $j=i$, since their Alexander gradings are $0$. We can see that $\widehat{HF}(S^3)\cong\F_2$ is generated by $[b^1_1,0,3]+[a^1_2,0,3]$, so their Maslov gradings are set to be zero, and this fixes the Maslov grading of the other generators. Thus, the white vertices have the Maslov grading zero. 

Also, we apply a change of basis:
\begin{itemize}
\item $b^1_1\longrightarrow b^1_1+a^1_2=:B^1_1$,
\item $d^1_1\longrightarrow d^1_1+c^1_2=:D^1_1$,
\item $f_1\longrightarrow f_1+e_1=:F_1$,
\end{itemize}
where $(i,j)$ grading are preserved. This change of basis makes the differentials simpler. For example, 
\begin{align*}
\partial [a^1_1,i,j]&=([a^1_2,i-1,j]+[b^1_1,i-1,j])+([f_1,i,j-2]+[e_1,i,j-2])\\
&=[B^1_1,i-1,j]+[F_1,i,j-2],
\end{align*}
also,
\begin{align*}
\partial [B^1_1,i,j]&=\partial([b^1_1,i,j]+[a^1_2,i,j])\\
&=\partial[b^1_1,i,j]+\partial[a^1_2,i,j]\\
&=([a^1_3,i,j-1]+[b^1_2,i-1,j])+([a^1_3,i,j-1]+[b^1_2,i-1,j])\\
&=0.
\end{align*}
After the change of basis, we have $\CFK^{\infty}(K_1^{(3,4)})$ as in the right of Figure \ref{CFK-K_1+3,4}. 

\begin{figure}[h]
\centering
\begin{tabular}{cc}
\includegraphics[scale=0.3]{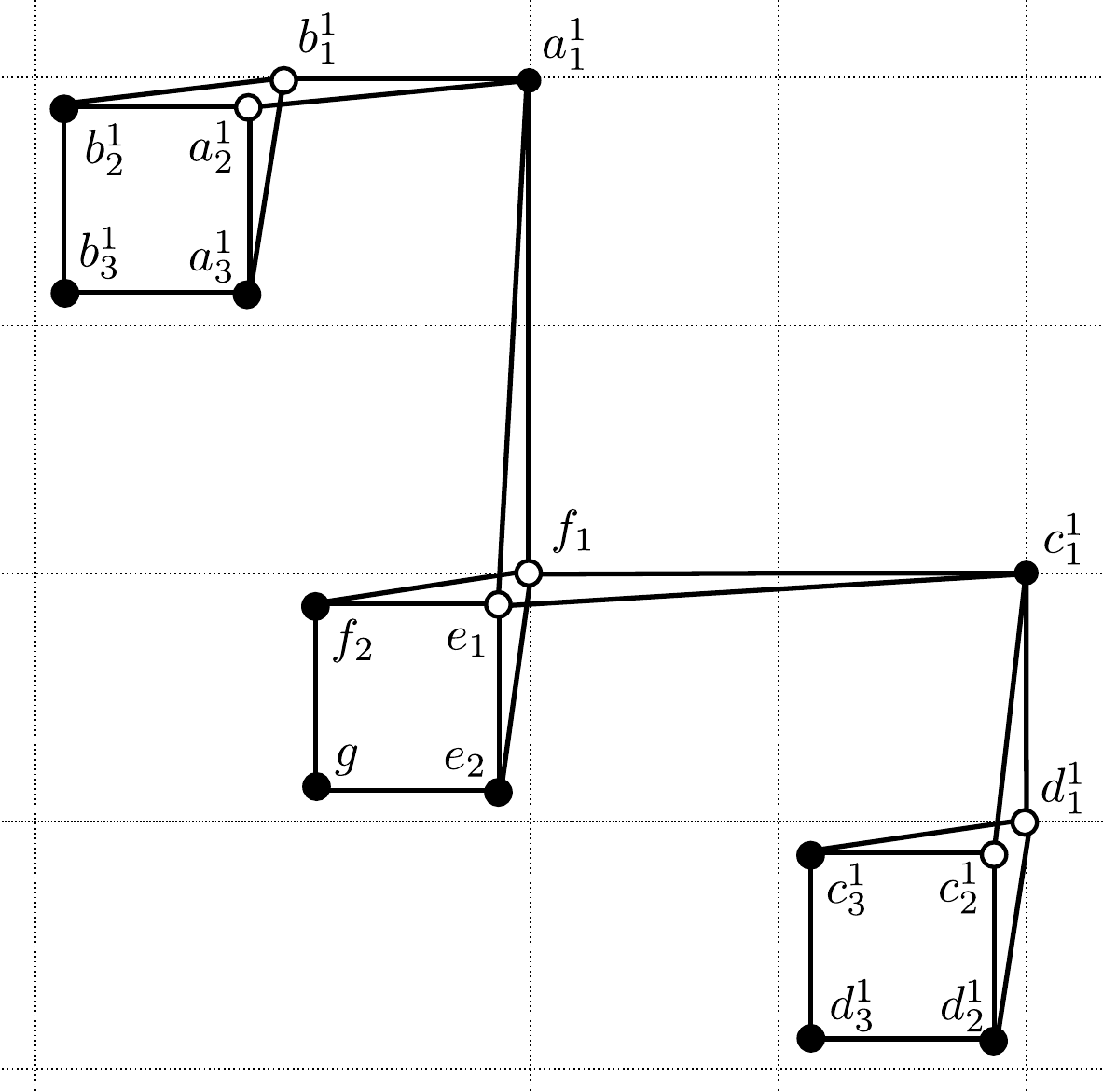}
&\includegraphics[scale=0.3]{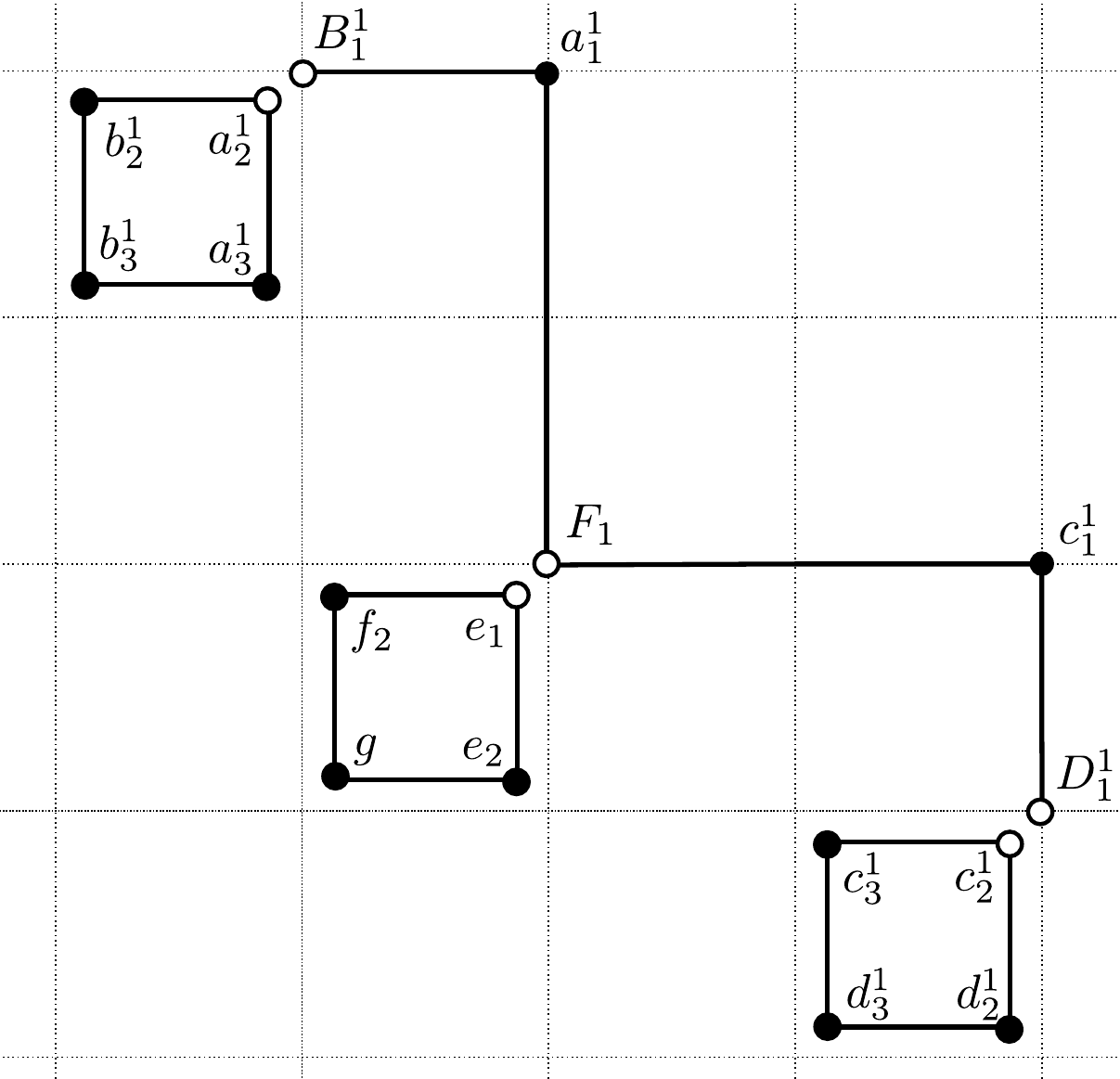}
\end{tabular}
\caption{The full knot Floer complex $\CFK^{\infty}(K_1^{(3,4)})$. The information of coordinates is omitted. Left is the complex derived from Figure \ref{1,1universal-K_1+3,4}. By applying the change of basis, right is obtained.}
\label{CFK-K_1+3,4}
\end{figure}

%%%%%%%%%%%%%%%%%%%%%%%%%%%%%%%%%%%%%%%%%%%%%%%%%%%%%%%%%%%%%%%%%%%%
\subsection{$\CFK^{\infty}(K_2^{(3,7)})$ (the case $q=3\cdot 2+1,\ n=2$)}
The method for calculating $\CFK^{\infty}(K_2^{(3,7)})$ is the same as in the previous subsection.

Figure \ref{1,1universal-K_2+3,7} shows the universal cover of a $(1,1)$--diagram of $K_2^{(3,7)}$. The intersection points of $\tilde{\alpha}$ and $\tilde{\beta}$ are labeled by 
\begin{align*}
&b^2_1,b^2_2,\ldots,b^2_5,a^2_5,\ldots,a^2_2,a^2_1,b^1_1,\ldots,b^1_5,a^1_5,\ldots,a^1_1,f_1,f_2,\ldots,f_4,g,e_4,\ldots,e_2,e_1,\\
&c^1_1,c^1_2,\ldots,c^1_5,d^1_5,\ldots,d^1_2,d^1_1,c^2_1,\ldots,c^2_5,d^2_5,\ldots,d^2_1 
\end{align*}
along $\tilde{\beta}$ from left-upper to right-lower. 

\begin{figure}[h]
\centering
\begin{tabular}{c}
\includegraphics[scale=0.4]{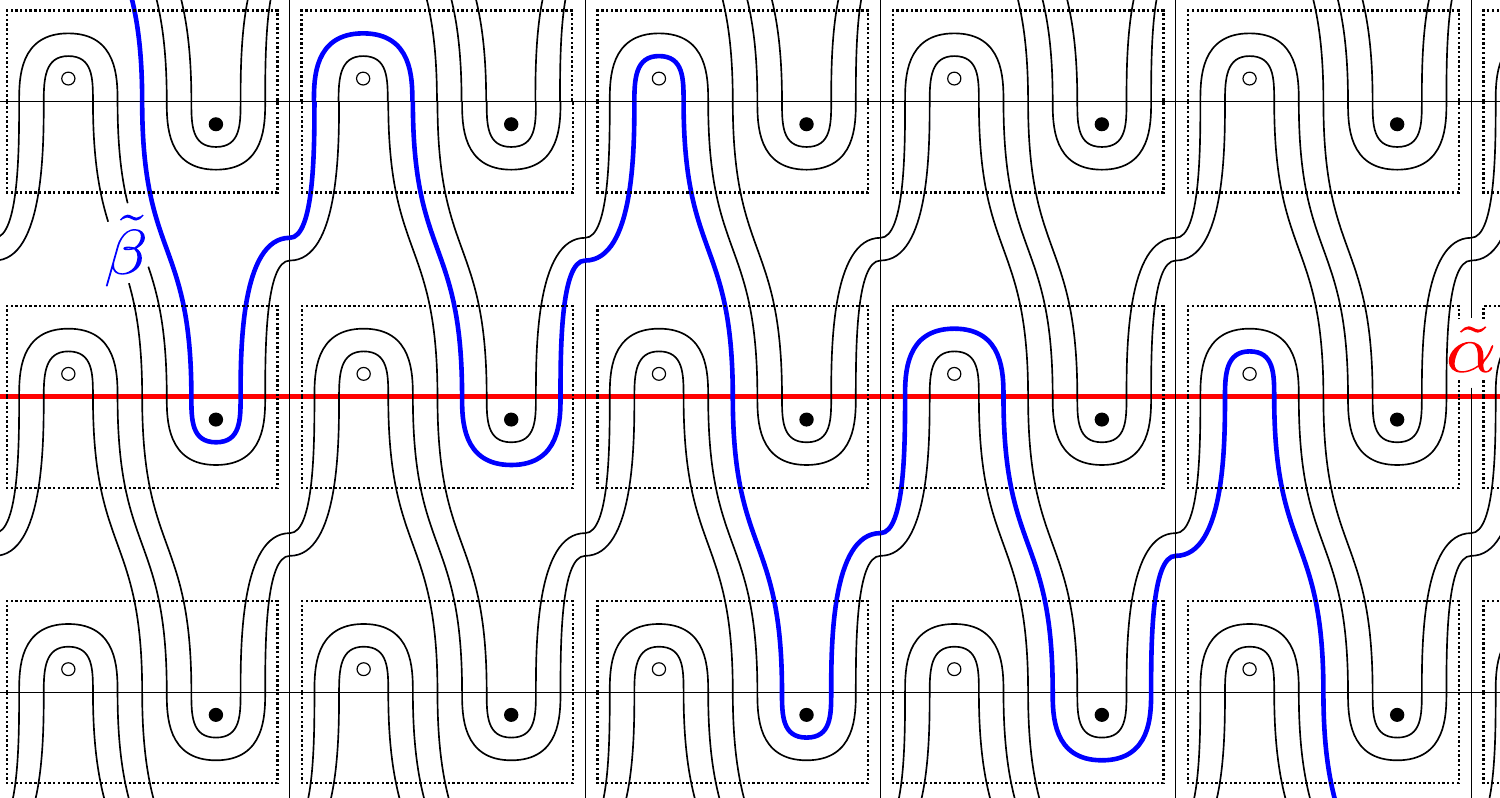}\\
$\downarrow$\\
\includegraphics[scale=0.4]{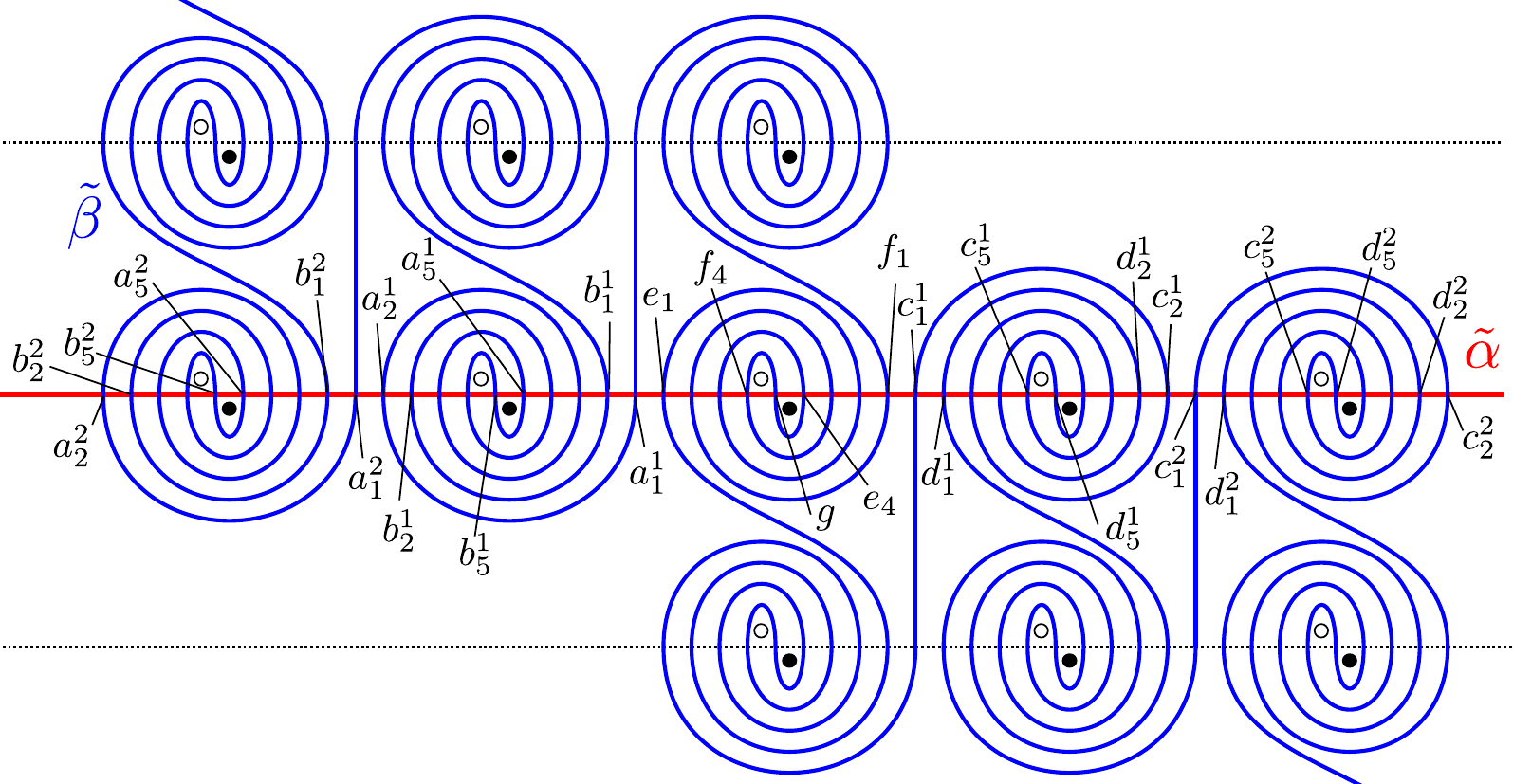}
\end{tabular}
\caption{The universal cover of a $(1,1)$--diagram of $K_2^{(3,7)}$. Bottom is obtained from top by twisting two points twice.}
\label{1,1universal-K_2+3,7}
\end{figure}

We list up all Whitney disks contributing to differentials of $\CFK^{\infty}(K_n^{(3,3k+1)})$ as in Table \ref{differential-K_2^{3,7}}. Also, the Alexander gradings of generators are given as in Table \ref{Alexander-K_2^{3,7}}.

\begin{table}[h]
\begingroup
\renewcommand{\arraystretch}{1.3}
\begin{tabular}{|c|c|c||c|c|c|}
\hline
from & to & base point & from & to & base point\\
\hline
\multirow{4}{*}{$a^1_1$} & $a^1_2$ & $\bullet$ & \multirow{4}{*}{$c^1_1$} & $c^1_2$ & $\circ$ \\
& $b^1_1$ & $\bullet$ && $d^1_1$ & $\circ$ \\ 
& $f_1$ & $\circ\circ$ && $f_1$ & $\bullet\bullet$\\
& $e_1$ & $\circ\circ$ && $e_1$ & $\bullet\bullet$\\
\hline
\multirow{4}{*}{$a^2_1$} & $a^2_2$ & $\bullet$ & \multirow{4}{*}{$c^2_1$} & $c^2_2$ & $\circ$ \\
& $b^2_1$ & $\bullet$ && $d^2_1$ & $\circ$ \\ 
& $a^1_2$ & $\circ\circ$ && $c^1_2$ & $\bullet\bullet$\\
& $b^1_1$ & $\circ\circ$ && $d^1_1$ & $\bullet\bullet$\\
\hline
\multirow{2}{*}{$a^i_3\ (i=1,2)$} & $a^i_4$ & $\bullet$ & \multirow{2}{*}{$c^i_3\ (i=1,2)$} & $c^i_4$ & $\circ$\\
& $b^i_3$ & $\bullet$ && $d^i_3$ & $\circ$\\
\hline
\multirow{2}{*}{\begin{tabular}{c}$a^i_{2j}$\\$(i=1,2,\ j=1,2)$\end{tabular}} & $a^i_{2j+1}$ & $\circ$ & \multirow{2}{*}{\begin{tabular}{c}$c^i_{2j}$\\$(i=1,2,\ j=1,2)$\end{tabular}} & $c^i_{2j+1}$ & $\bullet$\\
& $b^i_{2j}$ & $\bullet$ && $d^i_{2j}$ & $\circ$\\
\hline
$a^i_5\ (i=1,2)$ & $b^i_5$ & $\bullet$ & $c^i_5\ (i=1,2)$ & $d^i_5$ & $\circ$\\
\hline
\multirow{2}{*}{\begin{tabular}{c}$b^i_{2j-1}$\\$(i=1,2,\ j=1,2)$\end{tabular}} & $a^i_{2j+1}$ & $\circ$ & \multirow{2}{*}{\begin{tabular}{c}$d^i_{2j-1}$\\$(i=1,2,\ j=1,2)$\end{tabular}} & $c^i_{2j+1}$ & $\bullet$ \\
& $b^i_{2j}$ & $\bullet$ && $d^i_{2j}$ & $\circ$\\
\hline
\multirow{2}{*}{$b^i_2\ (i=1,2)$} & $a^i_4$ & $\circ$ & \multirow{2}{*}{$d^i_2\ (i=1,2)$} & $c^i_4$ & $\bullet$ \\
& $b^i_3$ & $\circ$ && $d^i_3$ & $\bullet$\\
\hline
$b^i_4\ (i=1,2)$ & $b^i_5$ & $\circ$ & $d^i_4\ (i=1,2)$ & $d^i_5$ & $\bullet$\\
\hline
\multirow{2}{*}{$f_{2j-1}\ (j=1,2)$} & $f_{2j}$ & $\bullet$ & \multirow{2}{*}{$e_{2j-1}\ (j=1,2)$} & $f_{2j}$ & $\bullet$\\
& $e_{2j}$ & $\circ$ && $e_{2j}$ & $\circ$\\
\hline 
\multirow{2}{*}{$f_2$} & $f_3$ & $\circ$ & \multirow{2}{*}{$e_2$} & $f_3$ & $\bullet$\\
& $e_3$ & $\circ$ && $e_3$ & $\bullet$\\
\hline
$f_4$ & $g$ & $\circ$ & $e_4$ & $g$ & $\bullet$\\
\hline
\end{tabular}
\endgroup
\caption{The list of Whitney disks contributing to differentials of $\CFK^{\infty}(K_2^{(3,7)})$.}
\label{differential-K_2^{3,7}}
\end{table}

\begin{table}[h]
\begingroup
\renewcommand{\arraystretch}{1.2}
\begin{tabular}{|c|c||c|c|}
\hline
Alexander grading & generators & Alexander grading & generators\\
\hline
$7$ & $b^2_2,\ b^2_4$ & $-7$ & $d^2_2,\ d^2_4$\\
\hline
$6$ & $a^2_2,\ a^2_4,\ b^2_1,\ b^2_3,\ b^2_5$ & $-6$ & $c^2_2,\ c^2_4,\ d^2_1,\ d^2_3,\ d^2_5$\\
\hline
$5$ & $a^2_1,\ a^2_3,\ a^2_5$ & $-5$ & $c^2_1,\ c^2_3,\ c^2_5$\\
\hline
$4$ & $b^1_2,\ b^1_4$ & $-4$ & $d^1_2,\ d^1_4$\\
\hline
$3$ & $a^1_2,\ a^1_4,\ b^1_1,\ b^1_3,\ b^1_5$ & $-3$ & $c^1_2,\ c^1_4,\ d^1_1,\ d^1_3,\ d^1_5$\\
\hline
$2$ & $a^1_1,\ a^1_3,\ a^1_5$ & $-2$ & $c^1_1,\ c^1_3,\ c^1_5$\\
\hline
$1$ & $f_2,\ f_4$ & $-1$ & $e_2,\ e_4$\\
\hline
$0$ & $g,\ f_1,\ f_3,\ e_1,\ e_3$ & & \\
\hline
\end{tabular}
\endgroup
\caption{The Alexander gradings of the generators of $\CFK^{\infty}(K_2^{(3,7)})$.} 
\label{Alexander-K_2^{3,7}}
\end{table}

By applying a change of basis: 
\begin{itemize}
\item $b^1_1\longrightarrow b^1_1+a^1_2=: B^1_1$, $b^2_1\longrightarrow b^2_1+a^2_2=: B^2_1$,
\item $b^1_3\longrightarrow b^1_3+a^1_4=: B^1_3$, $b^2_3\longrightarrow b^2_3+a^2_4=: B^2_3$,
\item $d^1_1\longrightarrow d^1_1+c^1_2=: D^1_1$, $d^2_1\longrightarrow d^2_1+c^2_2=: D^2_1$,
\item $d^1_3\longrightarrow d^1_3+c^1_4=: D^1_3$, $d^2_3\longrightarrow d^2_3+c^2_4=: D^2_3$,
\item $f_1\longrightarrow f_1+e_1=: F_1$, $f_3\longrightarrow f_3+e_3=: F_3$,
\end{itemize}
we obtain $\CFK^{\infty}(K_2^{(3,7)})$ as in Figure \ref{CFK-K_2+3,7}. 

\begin{figure}[h]
\centering
\includegraphics[scale=0.45]{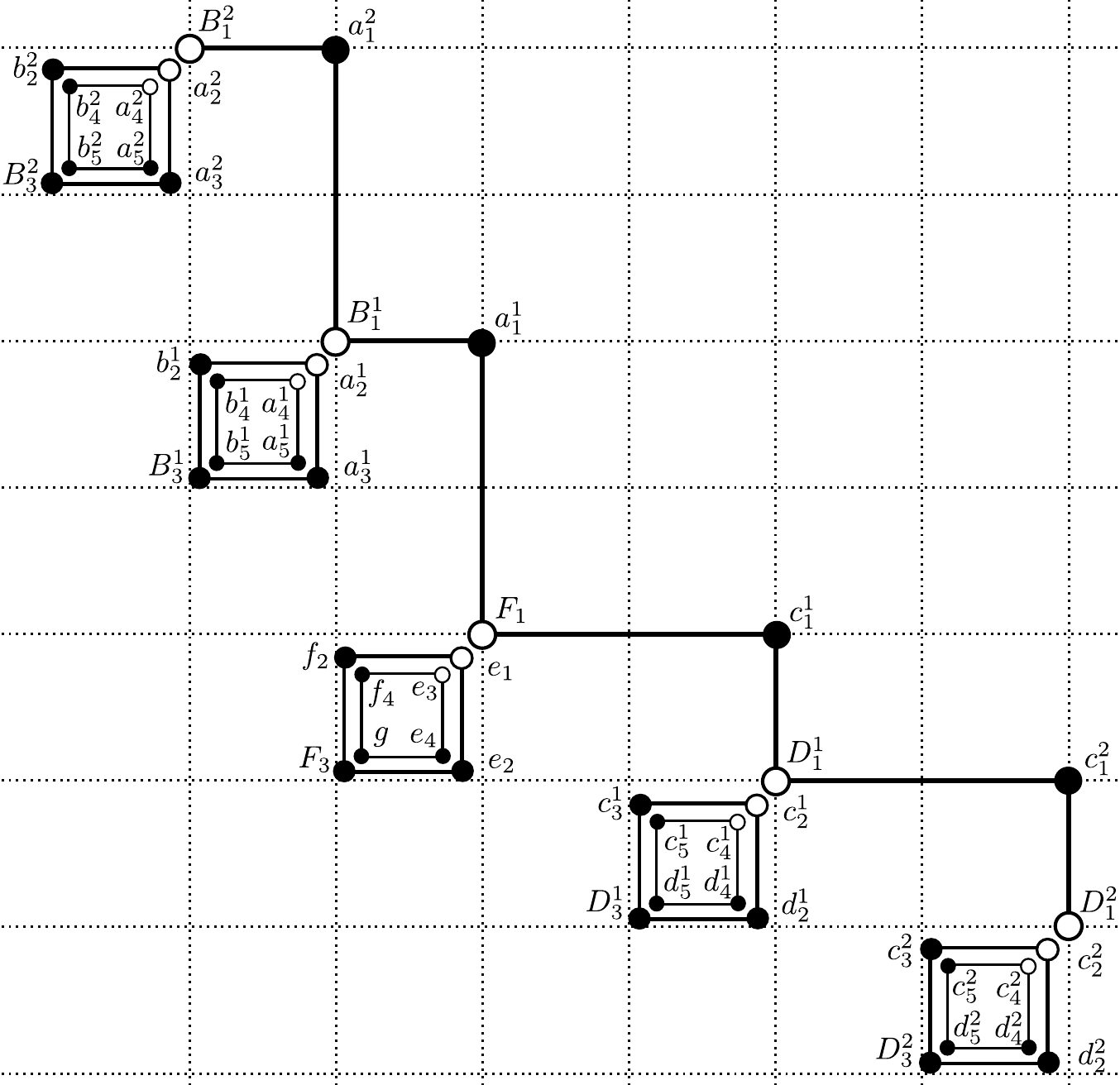}
\caption{The full knot Floer complex $\CFK^{\infty}(K_2^{(3,7)})$.}
\label{CFK-K_2+3,7}
\end{figure}

%%%%%%%%%%%%%%%%%%%%%%%%%%%%%%%%%%%%%%%%%%%%%%%%%%
\subsection{The general case $q=3k+1$}
Figure \ref{1,1universal-K_n3,3k+1} shows the universal cover of a $(1,1)$--diagram of $K_n^{(3,3k+1)}$, and Figures \ref{1,1universal-K_n3,3k+1-localleftright} and \ref{1,1universal-K_n3,3k+1-localcenter} represent enlarged figures of local regions in Figure \ref{1,1universal-K_n3,3k+1}. 

\begin{figure}[h]
\hspace*{-1.5cm}
\begin{tabular}{c}
\includegraphics[scale=0.5]{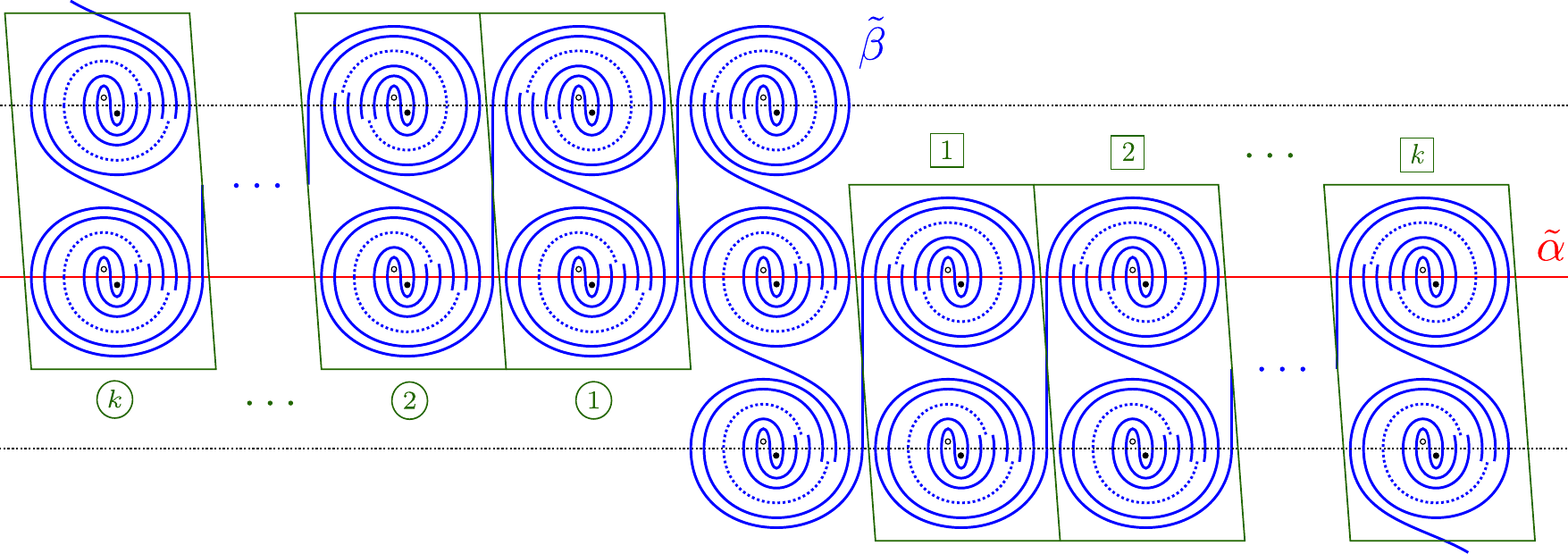}
\end{tabular}
\caption{The universal cover of a $(1,1)$--diagram of $K_n^{(3,3k+1)}$.}
\label{1,1universal-K_n3,3k+1}
\end{figure}

\begin{figure}[h]
\centering
\begin{tabular}{c|c}
\includegraphics[scale=0.385]{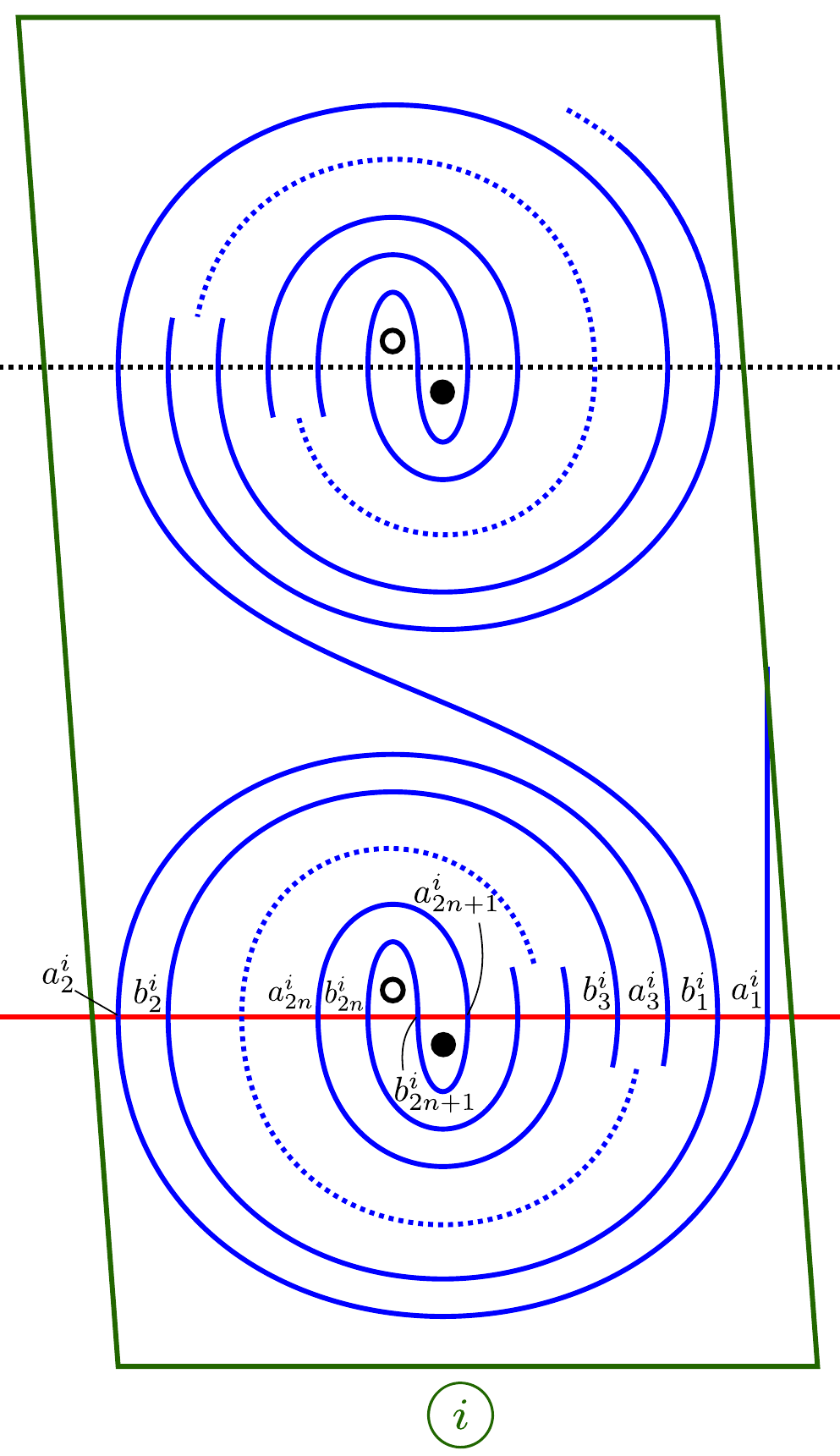}
&\includegraphics[scale=0.385]{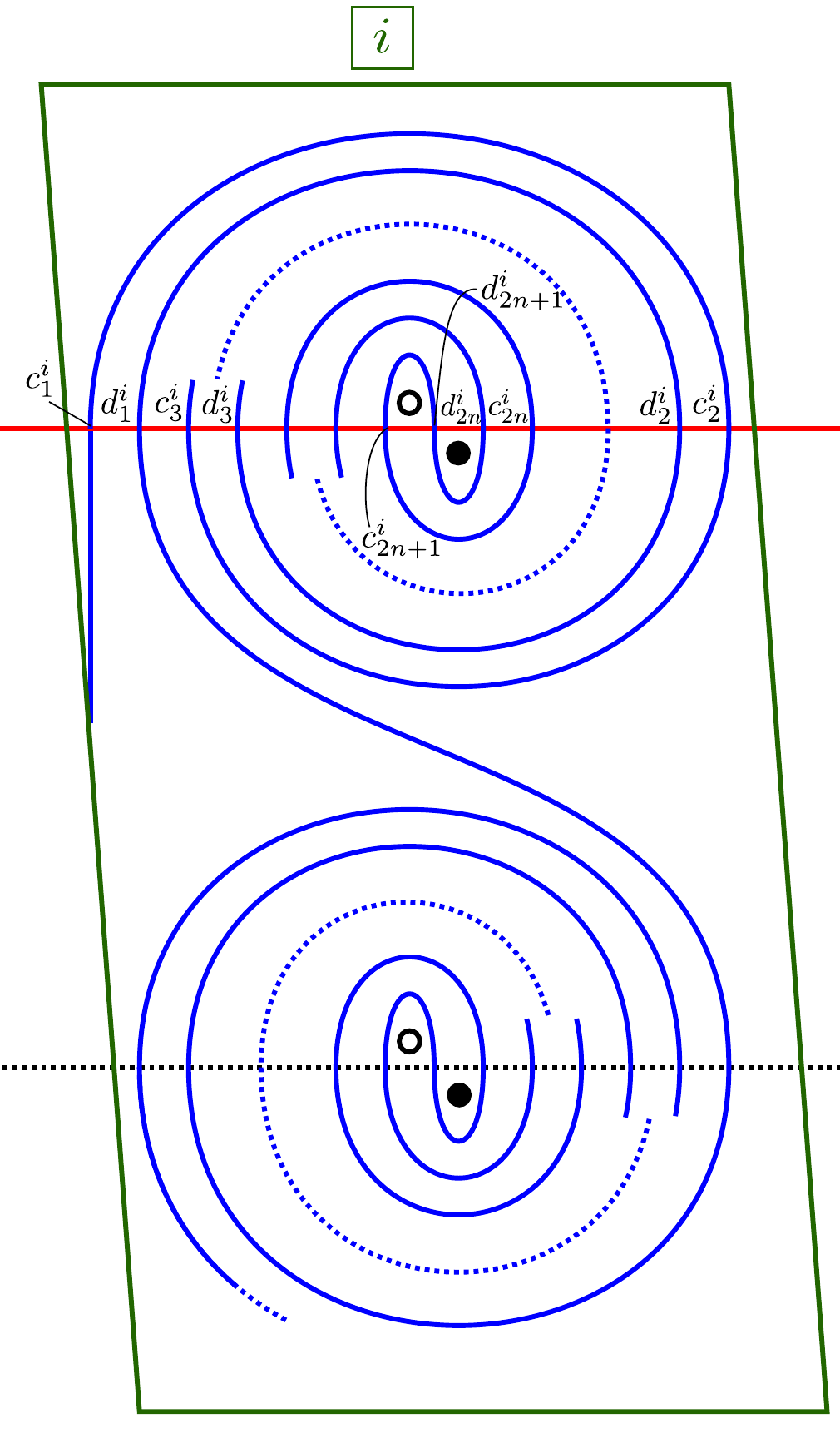}
\end{tabular}
\caption{Left (resp.right) is the $i$-th (green) box from the middle to the left (resp. right) in Figure \ref{1,1universal-K_n3,3k+1}. }
\label{1,1universal-K_n3,3k+1-localleftright}
\end{figure}

\begin{figure}[h]
\centering
\includegraphics[scale=0.4]{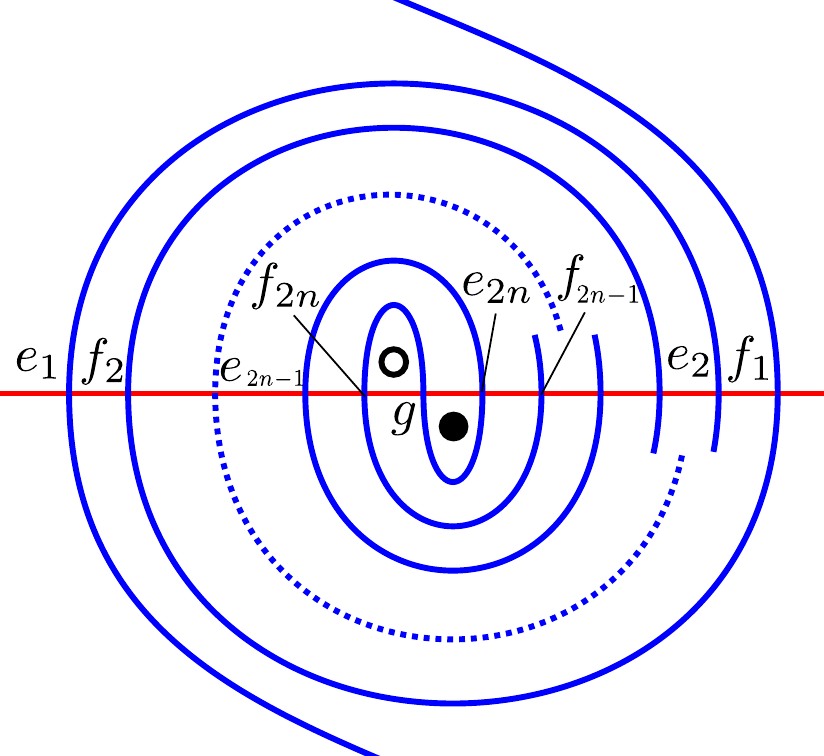}
\caption{The enlarged figure of the middle part in Figure \ref{1,1universal-K_n3,3k+1}. }
\label{1,1universal-K_n3,3k+1-localcenter}
\end{figure}

The intersection points of $\tilde{\alpha}$ and $\tilde{\beta}$ are labeled by 
\begin{align*}
&b^k_1,b^k_2,\ldots,b^k_{2n+1},a^k_{2n+1},a^k_{2n},\ldots,a^k_1,\\
&b^{k-1}_1,b^{k-1}_2,\ldots,b^{k-1}_{2n+1},a^{k-1}_{2n+1},a^{k-1}_{2n},\ldots,a^{k-1}_1,\\
&\vdots\\
&b^1_1,b^1_2,\ldots,b^1_{2n+1},a^1_{2n+1},a^1_{2n},\ldots,a^1_1,\\
&f_1,f_2,\ldots,f_{2n},g,e_{2n},e_{2n-1},\ldots,e_1\\
&c^1_1,c^1_2,\ldots,c^1_{2n+1},d^1_{2n+1},d^1_{2n},\ldots,d^1_1,\\
&c^2_1,c^2_2,\ldots,c^2_{2n+1},d^2_{2n+1},d^2_{2n},\ldots,d^1_1,\\
&\vdots\\
&c^k_1,c^k_2,\ldots,c^k_{2n+1},d^k_{2n+1},d^k_{2n},\ldots,d^k_1
\end{align*}
along $\tilde{\beta}$ from left-upper to right-lower (see Figures \ref{1,1universal-K_n3,3k+2-localleftright}, \ref{1,1universal-K_n3,3k+2-localcenter}).

In Table \ref{differential-Alexander-K_n^{3,3k+2}}, 
the Whitney disks contributing to differentials of $\CFK^{\infty}(K_n^{(3,3k+1)})$ and the Alexander gradings of generators are given as in Table \ref{differential-Alexander-K_n^{3,3k+1}}. Thus, we can see that $[g,i,i],\ [f_{{\rm odd}},i,i],\ [e_{{\rm odd}},i,i]$ lie on the line $j=i$. (See the caption of Table \ref{differential-Alexander-K_n^{3,3k+1}} for the meaning of subscripts `odd', `even'.)

\begin{table}[h]
\hspace*{-2cm}
\begingroup
\renewcommand{\arraystretch}{1.3}
\begin{tabular}{|c|c|c||c|c|c|}
\hline
from & to & base point & from & to & base point\\
\hline
\multirow{4}{*}{$a^1_1$} & $a^1_2$ & $\bullet$ & \multirow{4}{*}{$c^1_1$} & $c^1_2$ & $\circ$ \\
& $b^1_1$ & $\bullet$ && $d^1_1$ & $\circ$ \\ 
& $f_1$ & $\circ\circ$ && $f_1$ & $\bullet\bullet$\\
& $e_1$ & $\circ\circ$ && $e_1$ & $\bullet\bullet$\\
\hline
\multirow{4}{*}{$a^i_1\ (i=2,\ldots,k)$} & $a^i_2$ & $\bullet$ & \multirow{4}{*}{$c^i_1\ (i=2,\ldots,k)$} & $c^i_2$ & $\circ$ \\
& $b^i_1$ & $\bullet$ && $d^i_1$ & $\circ$ \\ 
& $a^{i-1}_2$ & $\circ\circ$ && $c^{i-1}_2$ & $\bullet\bullet$\\
& $b^{i-1}_1$ & $\circ\circ$ && $d^{i-1}_1$ & $\bullet\bullet$\\
\hline
\multirow{2}{*}{\begin{tabular}{c}$a^i_{2j-1}$\\ $(i=1,\ldots,k,\ j=2,\ldots,n)$\end{tabular}} & $a^i_{2j}$ & $\bullet$ & \multirow{2}{*}{\begin{tabular}{c}$c^i_{2j-1}$\\ $(i=1,\ldots,k,\ j=2,\ldots,n)$\end{tabular}} & $c^i_{2j}$ & $\circ$\\
& $b^i_{2j-1}$ & $\bullet$ && $d^i_{2j-1}$ & $\circ$\\
\hline
\multirow{2}{*}{\begin{tabular}{c}$a^i_{2j}$\\$(i=1,\ldots,k,\ j=1,\ldots,n)$\end{tabular}} & $a^i_{2j+1}$ & $\circ$ & \multirow{2}{*}{\begin{tabular}{c}$c^i_{2j}$\\$(i=1,\ldots,k,\ j=1,\ldots,n)$\end{tabular}} & $c^i_{2j+1}$ & $\bullet$\\
& $b^i_{2j}$ & $\bullet$ && $d^i_{2j}$ & $\circ$\\
\hline
$a^i_{2n+1}\ (i=1,\ldots,k)$ & $b^i_{2n+1}$ & $\bullet$ & $c^i_{2n+1}\ (i=1,\ldots,k)$ & $d^i_{2n+1}$ & $\circ$\\
\hline
\multirow{2}{*}{\begin{tabular}{c}$b^i_{2j-1}$\\$(i=1,\ldots,k,\ j=1,\ldots,n)$\end{tabular}} & $a^i_{2j+1}$ & $\circ$ & \multirow{2}{*}{\begin{tabular}{c}$d^i_{2j-1}$\\$(i=1,\ldots,k,\ j=1,\ldots,n)$\end{tabular}} & $c^i_{2j+1}$ & $\bullet$ \\
& $b^i_{2j}$ & $\bullet$ && $d^i_{2j}$ & $\circ$\\
\hline
\multirow{2}{*}{\begin{tabular}{c}$b^i_{2j}$\\$(i=1,\ldots,k,\ j=1,\ldots,n-1)$\end{tabular}} & $a^i_{2j+2}$ & $\circ$ & \multirow{2}{*}{\begin{tabular}{c}$d^i_{2j}$\\$(i=1,\ldots,k,\ j=1,\ldots,n-1)$\end{tabular}} & $c^i_{2j+2}$ & $\bullet$ \\
& $b^i_{2j+1}$ & $\circ$ && $d^i_{2j+1}$ & $\bullet$\\
\hline
$b^i_{2n}\ (i=1,\ldots,k)$ & $b^i_{2n+1}$ & $\circ$ & $d^i_{2n}\ (i=1,\ldots,k)$ & $d^i_{2n+1}$ & $\bullet$\\
\hline
\multirow{2}{*}{$f_{2j-1}\ (j=1,\ldots,n)$} & $f_{2j}$ & $\bullet$ & \multirow{2}{*}{$e_{2j-1}\ (j=1,\ldots,n)$} & $f_{2j}$ & $\bullet$\\
& $e_{2j}$ & $\circ$ && $e_{2j}$ & $\circ$\\
\hline 
\multirow{2}{*}{$f_{2j}\ (j=1,\ldots,n-1)$} & $f_{2j+1}$ & $\circ$ & \multirow{2}{*}{$e_{2j}\ (j=1,\ldots,n-1)$} & $f_{2j+1}$ & $\bullet$\\
& $e_{2j+1}$ & $\circ$ && $e_{2j+1}$ & $\bullet$\\
\hline
$f_{2n}$ & $g$ & $\circ$ & $e_{2n}$ & $g$ & $\bullet$\\
\hline
\end{tabular}
\endgroup

\vspace*{5mm}

\begingroup
\renewcommand{\arraystretch}{1.3}
\begin{tabular}{|c|c||c|c|}
\hline
Alexander grading & generators & Alexander grading & generators\\
\hline
$3k+1$ & $b^k_{{\rm even}}$ & $-3k-1$ & $d^k_{{\rm even}}$\\
\hline
$3k$ & $a^k_{{\rm even}},\ b^k_{{\rm odd}}$ & $-3k$ & $c^k_{{\rm even}},\ d^k_{{\rm odd}}$\\
\hline
$3k-1$ & $a^k_{{\rm odd}}$ & $-3k+1$ & $c^k_{{\rm odd}}$\\
\hline
$3k-2$ & $b^{k-1}_{{\rm even}}$ & $-3k+2$ & $d^{k-1}_{{\rm even}}$\\
\hline
$\vdots$ & $\vdots$ & $\vdots$ & $\vdots$\\
\hline
$2$ & $a^1_{{\rm odd}}$ & $-2$ & $c^1_{{\rm odd}}$\\
\hline
$1$ & $f_{{\rm even}}$ & $-1$ & $e_{{\rm even}}$\\
\hline
$0$ & $g,\ f_{{\rm odd}},\ e_{{\rm odd}}$ & & \\
\hline
\end{tabular}
\endgroup
\caption{(Top) The list of Whitney disks contributing to differentials of $\CFK^{\infty}(K_n^{(3,3k+1)})$. (Bottom) The Alexander gradings of the generators of $\CFK^{\infty}(K_n^{(3,3k+1)})$. The subscript {\it even} indicates $2j\ (j=1,\ldots,n)$, for example, $b^k_{2j}\ (j=1,\ldots,n)$ has the Alexander grading $3k+1$. The subscript {\it odd} similarly indicates $2j-1\ (j=1,\ldots,n+1)$.} 
\label{differential-Alexander-K_n^{3,3k+1}}
\end{table}

Moreover, we apply a change of basis: 
\begin{itemize}
\item $b^i_{2j-1}\longrightarrow b^i_{2j-1}+a^i_{2j}=: B^i_{2j-1}$ for $i=1,\ldots,k$ and $j=1,\ldots,n$,
\item $d^i_{2j-1}\longrightarrow d^i_{2j-1}+c^i_{2j}=: D^i_{2j-1}$ for $i=1,\ldots,k$ and $j=1,\ldots,n$,
\item $f_{2j-1}\longrightarrow f_{2j-1}+e_{2j-1}=: F_{2j-1}$ for $j=1,\ldots,n$.
\end{itemize}
Then, we have $\CFK^{\infty}(K_n^{(3,3k+1)})$ as in Figure \ref{CFK-K_n3,3k+1-precise}. 

\begin{figure}[h]
\centering
\begin{tabular}{c}
\includegraphics[scale=0.7]{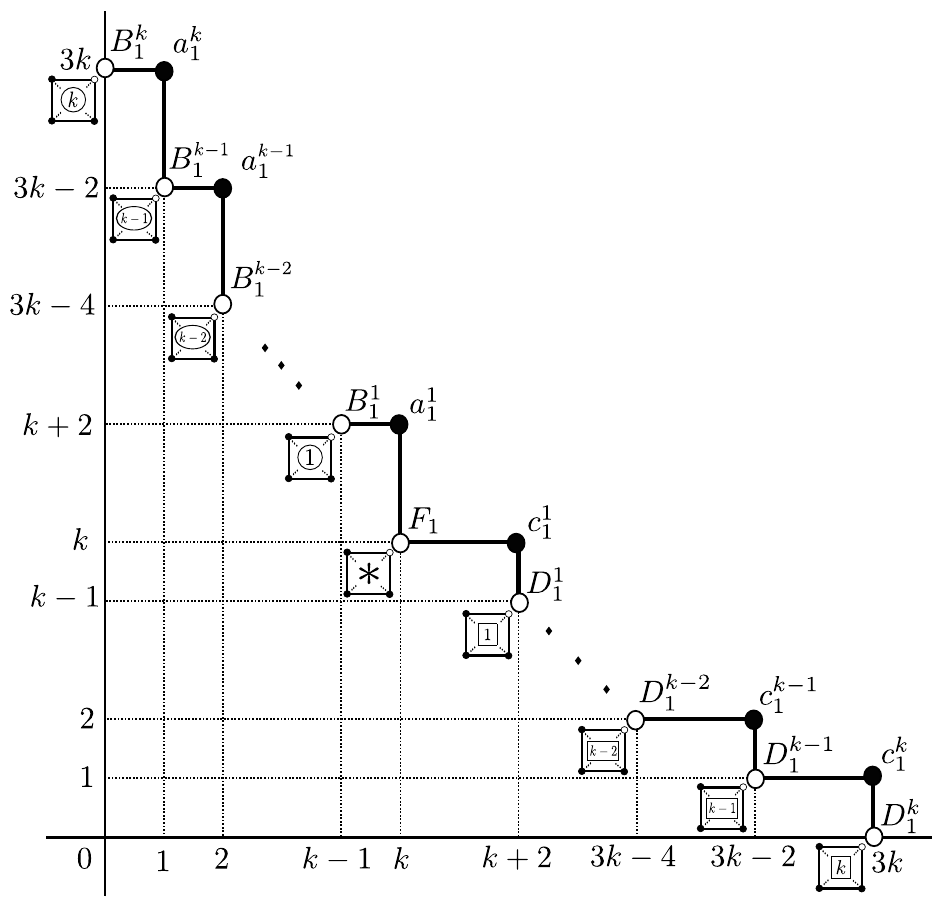}\\
\includegraphics[scale=0.4]{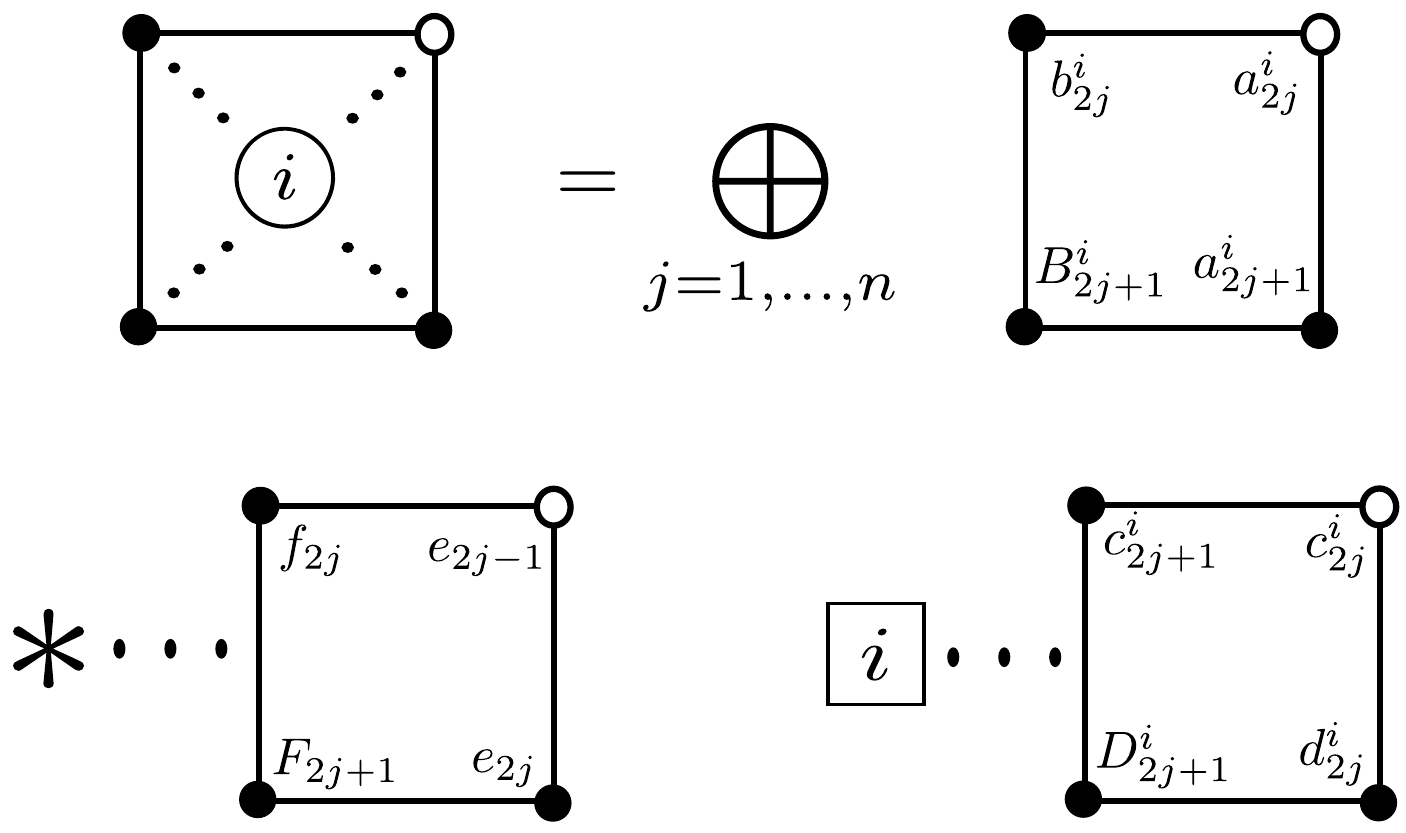}
\end{tabular}
\caption{The full knot Floer complex $\CFK^{\infty}(K_n^{(3,3k+1)})$. The box complex with \textcircled{\scriptsize $i$} represents the direct sum of $n$ box complexes consisting of $a^i_{2j}, a^i_{2j+1},b^i_{2j},B^i_{2j+1}\ (j=1,\ldots,n)$, where $B^j_{2n+1}:=b^i_{2n+1}$. The box with $*$ or a boxed number similarly represents the direct sum of $n$ box complexes (where $D^i_{2n+1}:=d^i_{2n+1}$ and $F_{2n+1}:=g$). Also, we can see that white vertices have the Maslov grading zero.}
\label{CFK-K_n3,3k+1-precise}
\end{figure}

%%%%%%%%%%%%%%%%%%%%%%%%%%%%%%%%%%%%%%%%%%%%%%%%%%%%%%%%%%%%%%%%%%%%%%%%%%%%%%%%%%%%%%%%%%%%%%%%%%%%%%%%%%%%
\section{The case $q=3k+2$}\label{calCFK3k+2}
In this section, we give a $(1,1)$--diagram of $K_n^{(3,3k+2)}$ and calculate $\CFK^{\infty}(K_n^{3,3k+2})$. The procedure to obtain them is similar to those in Sections \ref{(1,1)diagram} and \ref{calCFK3k+1}. 

Figure \ref{1,1diagram-K_n3,3k+2} shows that how to get a $(1,1)$--diagram of $K_n^{(3,3k+2)}$, see also Figure \ref{1,1diagram-K_n3,3k+1-1-4}, \ref{1,1diagram-K_n3,3k+1-5-6} and \ref{1,1diagram-K_n3,3k+1-afterDhentwists} in Section \ref{(1,1)diagram}. Moreover, the universal cover of this $(1,1)$--diagram is given as in Figure \ref{1,1universal-K_n3,3k+2}. 

The intersection points of $\tilde{\alpha}$ and $\tilde{\beta}$ are labeled by 
\begin{align*}
&a^k_1,a^k_2,\ldots,a^k_{2n+1},b^k_{2n+1},\ldots,b^k_1,\\
&\vdots\\
&a^1_1,a^1_2,\ldots,a^1_{2n+1},b^1_{2n+1},\ldots,b^1_1,\\
&e_1,e_2,\ldots,e_{2n+1},g,f_{2n+1},\ldots,f_1\\
&d^1_1,d^1_2,\ldots,d^1_{2n+1},c^1_{2n+1},\ldots,c^1_1,\\
&\vdots\\
&d^k_1,d^k_2,\ldots,d^k_{2n+1},c^k_{2n+1},\ldots,c^k_1
\end{align*}
along $\tilde{\beta}$ from left-upper to right-lower (see Figures \ref{1,1universal-K_n3,3k+1-localleftright}, \ref{1,1universal-K_n3,3k+1-localcenter}). Then, Table \ref{differential-Alexander-K_n^{3,3k+2}} shows the Whitney disks contributing to differentials of $\CFK^{\infty}(K_n^{(3,3k+2)})$ and the Alexander gradings of generators.

\begin{figure}[h]
\begin{tabular}{ccc}
\begin{minipage}{.50\textwidth}
\centering
\includegraphics[scale=0.24]{1,1diagram-K_n3,3k+1-1.pdf}
\end{minipage}
$\longrightarrow$
\begin{minipage}{.50\textwidth}
\centering
\includegraphics[scale=0.24]{1,1diagram-K_n3,3k+1-3.pdf}
\end{minipage}\\
$\swarrow$\\
\begin{minipage}{.50\textwidth}
\hspace*{-3mm}
\includegraphics[scale=0.24]{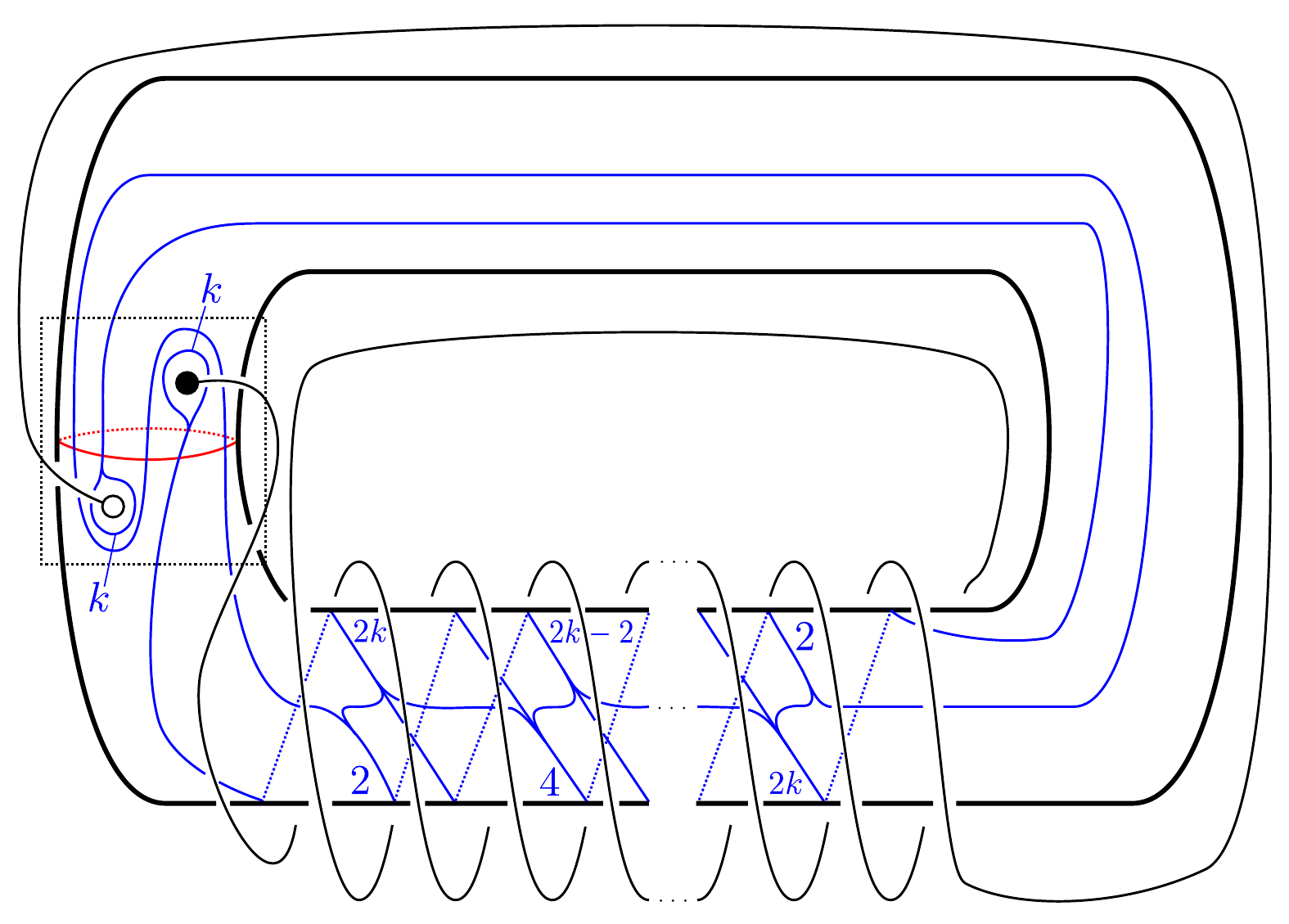}
\end{minipage}
$\longrightarrow$
\begin{minipage}{.50\textwidth}
\centering
\includegraphics[scale=0.24]{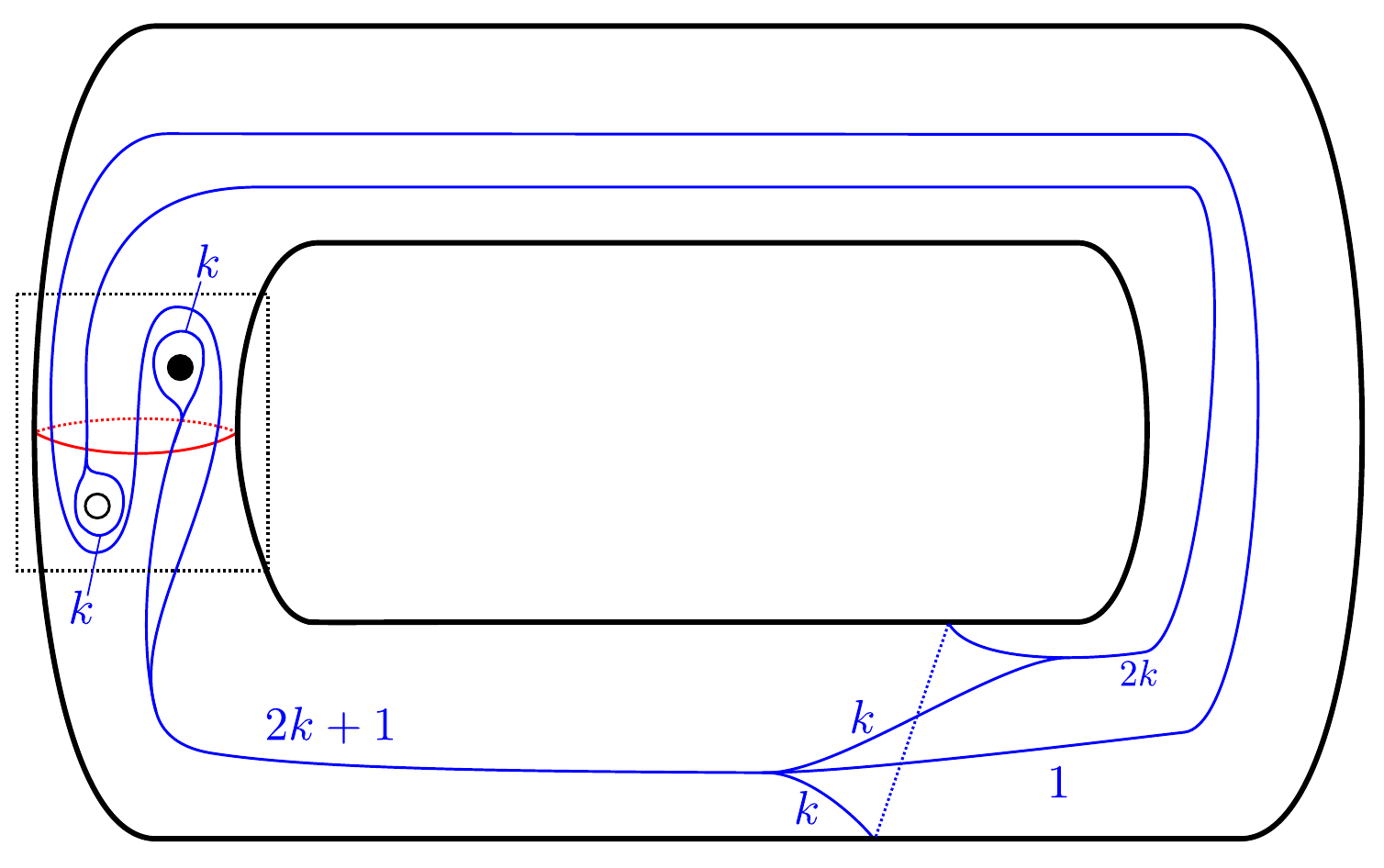}
\end{minipage}\\
\end{tabular}
\caption{How to move the arc $t$ to get a $(1,1)$--diagram of $K_n^{(3,3k+2)}$. The first move is exactly the same as the case $q=3k+1$ (the move from the top left to the bottom left in Figure \ref{1,1diagram-K_n3,3k+1-1-4}). 
The bottom left is a $(1,1)$--diagram of $K_0^{(3,3k+2)}=T(3,3k+2)$. Note that it matches the $(1,1)$--diagram of $T(3,3k+1)$ (the bottom right in Figure \ref{1,1diagram-K_n3,3k+1-1-4}) except for the region indicated the dotted box. 
The bottom right is after left handed Dehn twists along the $\alpha$ curve. 
As before, to obtain a $(1,1)$--diagram of $K_n^{(3,3k+2)}$, we need to add $n$-twists counter-clockwisely inside the dotted box.}
\label{1,1diagram-K_n3,3k+2}
\end{figure}

\begin{figure}[h]
\hspace*{-0.5cm}
\begin{tabular}{c}
\includegraphics[scale=0.45]{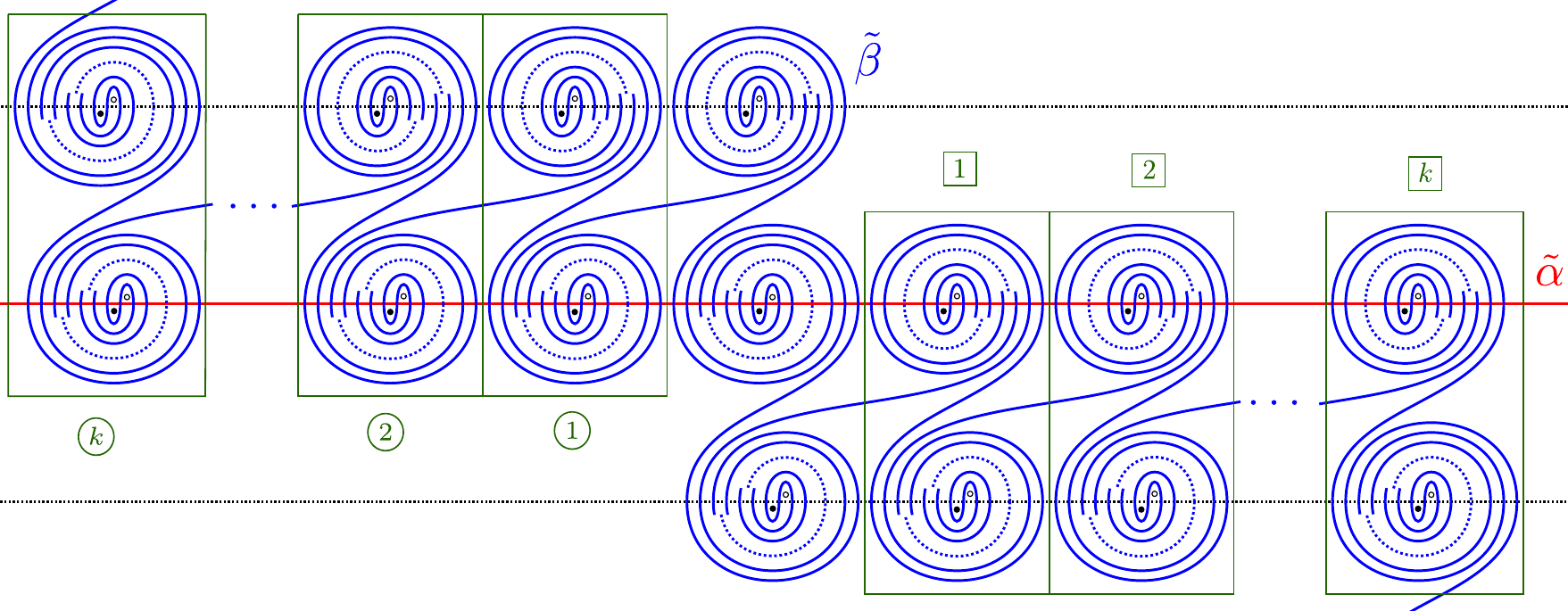}
\end{tabular}
\caption{The universal cover of a $(1,1)$--diagram of $K_n^{(3,3k+2)}$.}
\label{1,1universal-K_n3,3k+2}
\end{figure}

\begin{figure}[h]
\centering
\begin{tabular}{c|c}
\includegraphics[scale=0.385]{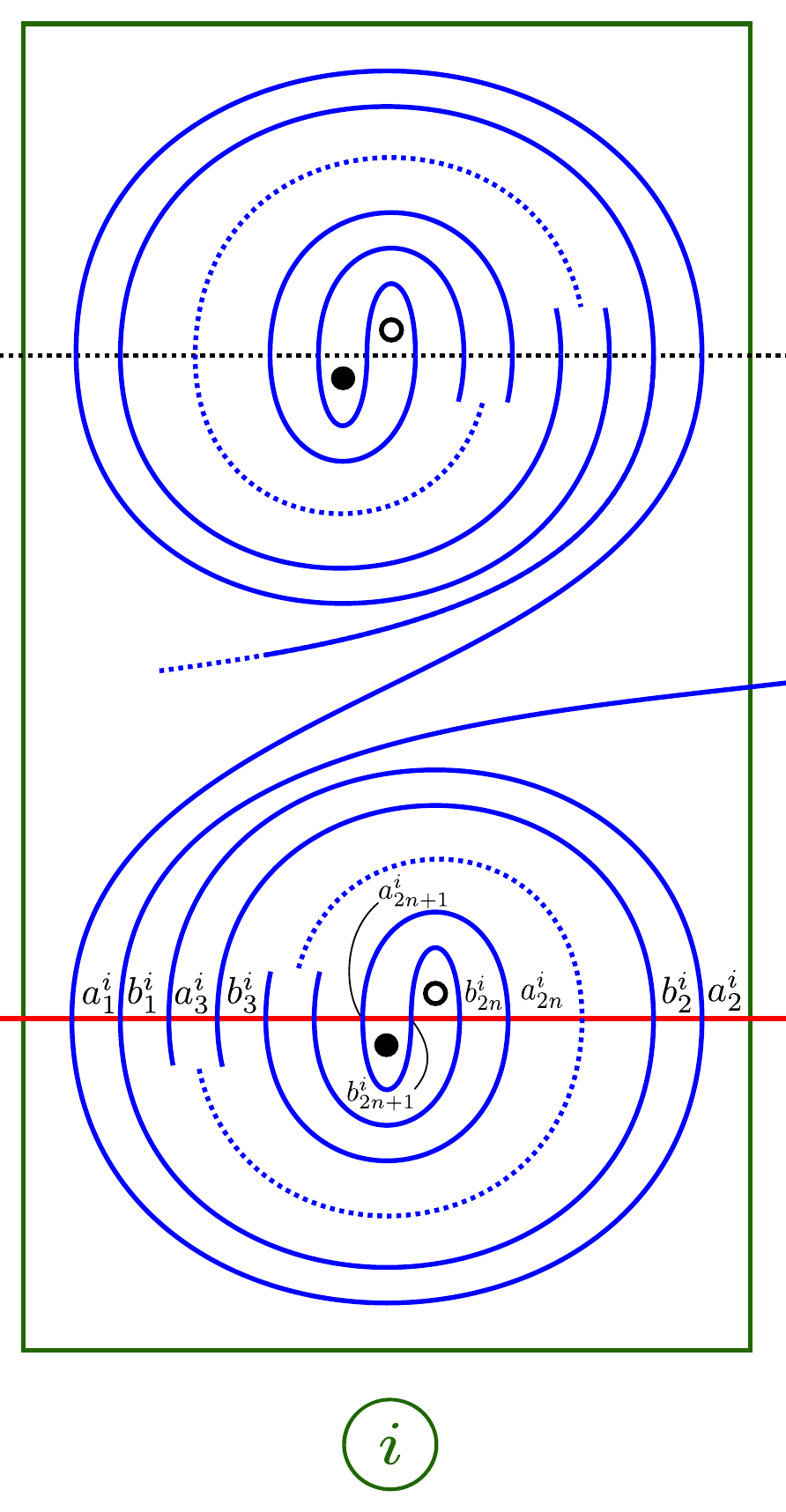}
&\includegraphics[scale=0.385]{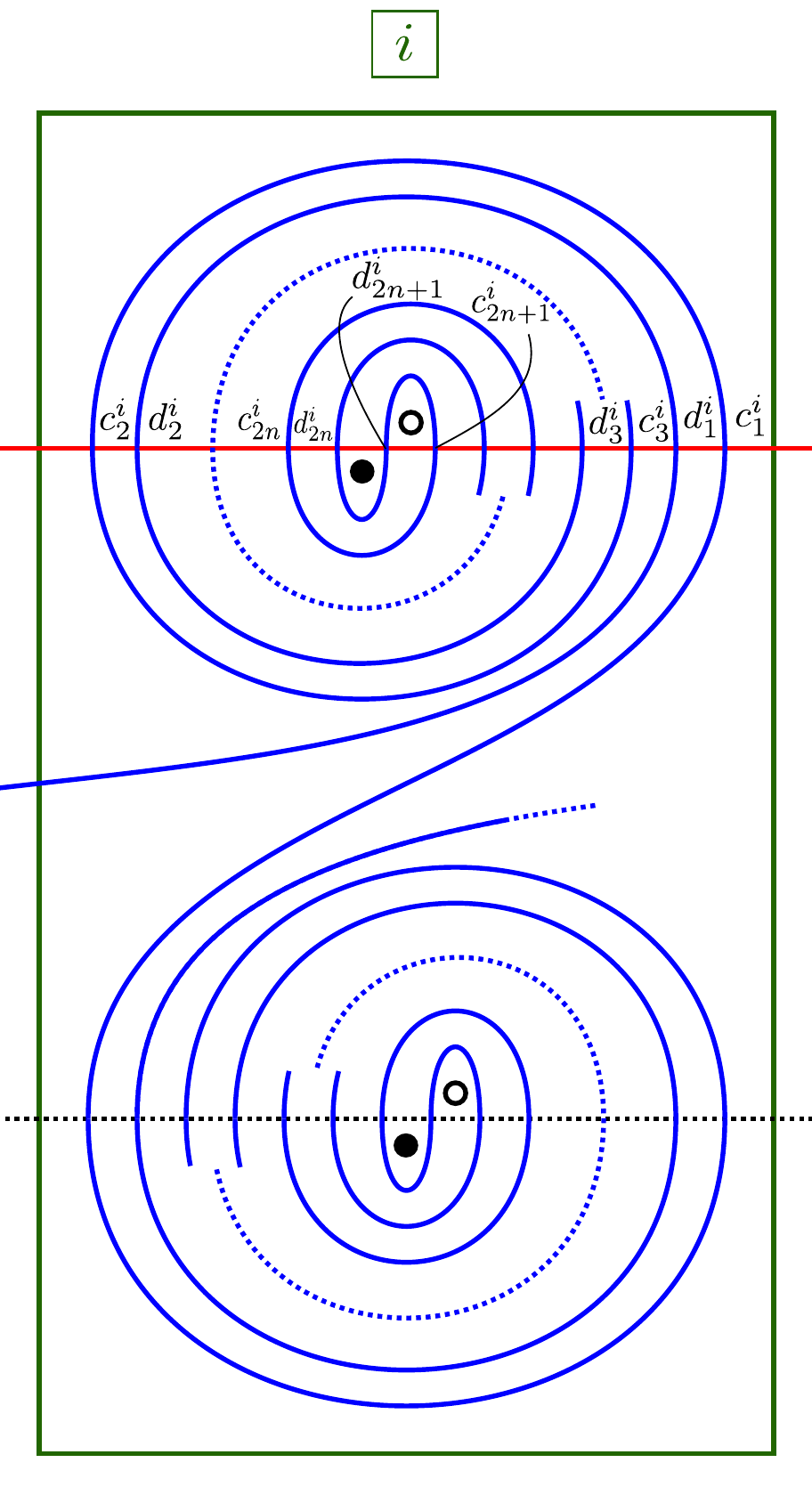}
\end{tabular}
\caption{Left (resp.  right) is the $i$-th (green) box from the middle to the left (resp. right) in Figure \ref{1,1universal-K_n3,3k+2}. }
\label{1,1universal-K_n3,3k+2-localleftright}
\end{figure}

\begin{figure}[h]
\centering
\includegraphics[scale=0.4]{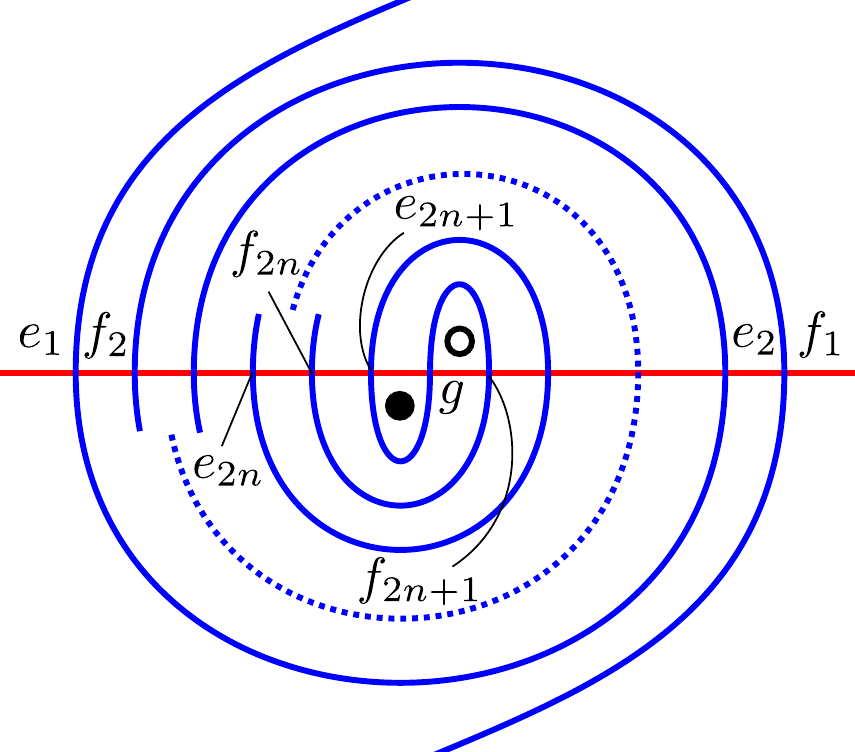}
\caption{The enlarged figure of the middle part in Figure \ref{1,1universal-K_n3,3k+2}. }
\label{1,1universal-K_n3,3k+2-localcenter}
\end{figure}

\begin{table}[h]
\hspace*{-2cm}
\begingroup
\renewcommand{\arraystretch}{1.3}
\begin{tabular}{|c|c|c||c|c|c|}
\hline
from & to & base point & from & to & base point\\
\hline
\multirow{2}{*}{$a^1_2$} & $a^1_1$ & $\bullet$ & \multirow{2}{*}{$c^1_2$} & $c^1_1$ & $\circ$ \\
& $e_1$ & $\circ\circ$ && $f_1$ & $\bullet\bullet$ \\ 
\hline
\multirow{2}{*}{$a^i_2\ (i=2,\ldots,k)$} & $a^1_1$ & $\bullet$ & \multirow{2}{*}{$c^i_2\ (i=2,\ldots,k)$} & $c^1_1$ & $\circ$ \\
& $a^{i-1}_1$ & $\circ\circ$ && $c^{i-1}_1$ & $\bullet\bullet$ \\ 
\hline
\multirow{2}{*}{\begin{tabular}{c}$a^i_{2j+1}$\\ $(i=1,\ldots,k,\ j=1,\ldots,n)$\end{tabular}} & $a^i_{2j}$ & $\circ$ & \multirow{2}{*}{\begin{tabular}{c}$c^i_{2j+1}$\\ $(i=1,\ldots,k,\ j=1,\ldots,n)$\end{tabular}} & $c^i_{2j}$ & $\bullet$\\
& $b^i_{2j-1}$ & $\circ$ && $d^i_{2j-1}$ & $\bullet$\\
\hline
\multirow{2}{*}{\begin{tabular}{c}$a^i_{2j}$\\$(i=1,\ldots,k,\ j=2,\ldots,n)$\end{tabular}} & $a^i_{2j-1}$ & $\bullet$ & \multirow{2}{*}{\begin{tabular}{c}$c^i_{2j}$\\$(i=1,\ldots,k,\ j=2,\ldots,n)$\end{tabular}} & $c^i_{2j-1}$ & $\circ$\\
& $b^i_{2j-2}$ & $\circ$ && $d^i_{2j-2}$ & $\bullet$\\
\hline
\multirow{2}{*}{$b^1_1$} & $a^1_1$ & $\bullet$ & \multirow{2}{*}{$d^1_1$} & $c^1_1$ & $\circ$ \\
& $e_1$ & $\circ\circ$ && $f_1$ & $\bullet\bullet$ \\ 
\hline
\multirow{2}{*}{$b^i_1\ (i=2,\ldots,k)$} & $a^i_1$ & $\bullet$ & \multirow{2}{*}{$d^i_1\ (i=1,\ldots,k)$} & $c^i_1$ & $\circ$ \\
& $a^{i-1}_1$ & $\circ\circ$ && $c^{i-1}_1$ & $\bullet\bullet$ \\ 
\hline
\multirow{2}{*}{\begin{tabular}{c}$b^i_{2j-1}$\\$(i=1,\ldots,k,\ j=2,\ldots,n+1)$\end{tabular}} & $a^i_{2j-1}$ & $\bullet$ & \multirow{2}{*}{\begin{tabular}{c}$c^i_{2j-1}$\\$(i=1,\ldots,k,\ j=2,\ldots,n+1)$\end{tabular}} & $c^i_{2j-1}$ & $\circ$\\
& $b^i_{2j-2}$ & $\circ$ && $d^i_{2j-2}$ & $\bullet$\\
\hline
\multirow{2}{*}{\begin{tabular}{c}$b^i_{2j}$\\$(i=1,\ldots,k,\ j=1,\ldots,n)$\end{tabular}} & $a^i_{2j}$ & $\bullet$ & \multirow{2}{*}{\begin{tabular}{c}$d^i_{2j}$\\$(i=1,\ldots,k,\ j=1,\ldots,n)$\end{tabular}} & $c^i_{2j}$ & $\circ$ \\
& $b^i_{2j-1}$ & $\bullet$ && $d^i_{2j-1}$ & $\circ$\\
\hline
\multirow{2}{*}{$f_{2j-1}\ (j=2,\ldots,n+1)$} & $f_{2j-2}$ & $\bullet$ & \multirow{2}{*}{$e_{2j-1}\ (j=2,\ldots,n+1)$} & $f_{2j-2}$ & $\circ$\\
& $e_{2j-2}$ & $\bullet$ && $e_{2j-2}$ & $\circ$\\
\hline 
\multirow{2}{*}{$f_{2j}\ (j=1,\ldots,n)$} & $f_{2j-1}$ & $\circ$ & \multirow{2}{*}{$e_{2j}\ (j=1,\ldots,n)$} & $f_{2j-1}$ & $\circ$\\
& $e_{2j-1}$ & $\bullet$ && $e_{2j-1}$ & $\bullet$\\
\hline
\multirow{2}{*}{$g$} & $f_{2n+1}$ & $\circ$ &&&\\
& $e_{2n+1}$ & $\bullet$ &&&\\
\hline
\end{tabular}
\endgroup

\vspace*{5mm}

\begingroup
\renewcommand{\arraystretch}{1.3}
\begin{tabular}{|c|c||c|c|}
\hline
Alexander grading & generators & Alexander grading & generators\\
\hline
$3k+1$ & $a^k_{{\rm odd}}$ & $-3k-1$ & $c^k_{{\rm odd}}$\\
\hline
$3k$ & $a^k_{{\rm even}},\ b^k_{{\rm odd}}$ & $-3k$ & $c^k_{{\rm even}},\ d^k_{{\rm odd}}$\\
\hline
$3k-1$ & $b^k_{{\rm even}}$ & $-3k+1$ & $d^k_{{\rm even}}$\\
\hline
$3k-2$ & $a^{k-1}_{{\rm odd}}$ & $-3k+2$ & $c^{k-1}_{{\rm odd}}$\\
\hline
$\vdots$ & $\vdots$ & $\vdots$ & $\vdots$\\
\hline
$2$ & $a^1_{{\rm even}}$ & $-2$ & $d^1_{{\rm even}}$\\
\hline
$1$ & $e_{{\rm odd}}$ & $-1$ & $f_{{\rm odd}}$\\
\hline
$0$ & $g,\ e_{{\rm even}},\ f_{{\rm even}}$ & & \\
\hline
\end{tabular}
\endgroup
\caption{(Top) The list of Whitney disks contributing to differentials of $\CFK^{\infty}(K_n^{(3,3k+2)})$. (Bottom) The Alexander gradings of the generators of $\CFK^{\infty}(K_n^{(3,3k+2)})$.} 
\label{differential-Alexander-K_n^{3,3k+2}}
\end{table}

By applying a change of basis: 
\begin{itemize}
\item $b^i_{2j-1}\longrightarrow b^i_{2j-1}+a^i_{2j}=: B^i_{2j-1}$ for $i=1,\ldots,k$ and $j=1,\ldots,n$,
\item $d^i_{2j-1}\longrightarrow d^i_{2j-1}+c^i_{2j}=: D^i_{2j-1}$ for $i=1,\ldots,k$ and $j=1,\ldots,n$,
\item $f_{2j}\longrightarrow f_{2j}+e_{2j}=: F_{2j}$ for $j=1,\ldots,n$,
\end{itemize}
we obtain $\CFK^{\infty}(K_n^{(3,3k+2)})$ as in Figure \ref{CFK-K_n3,3k+2-precise}.

\begin{figure}[h]
\centering
\begin{tabular}{c}
\includegraphics[scale=0.75]{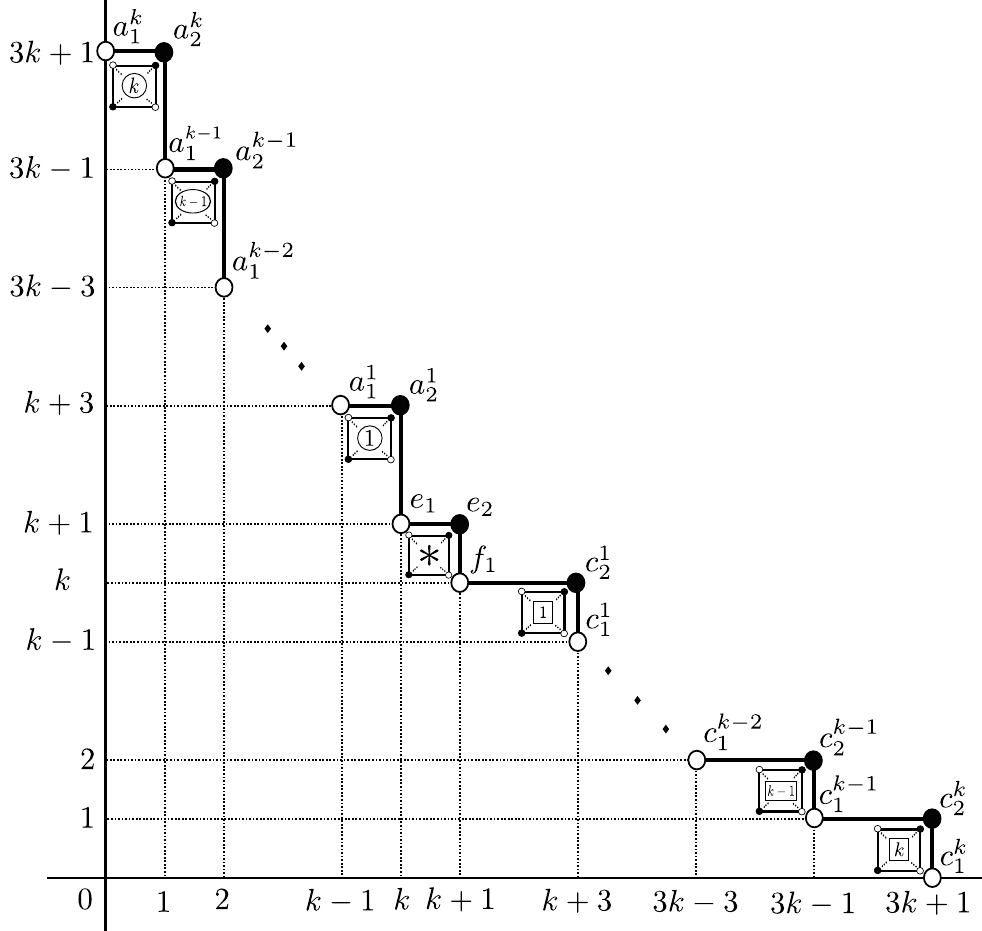}\\
\includegraphics[scale=0.4]{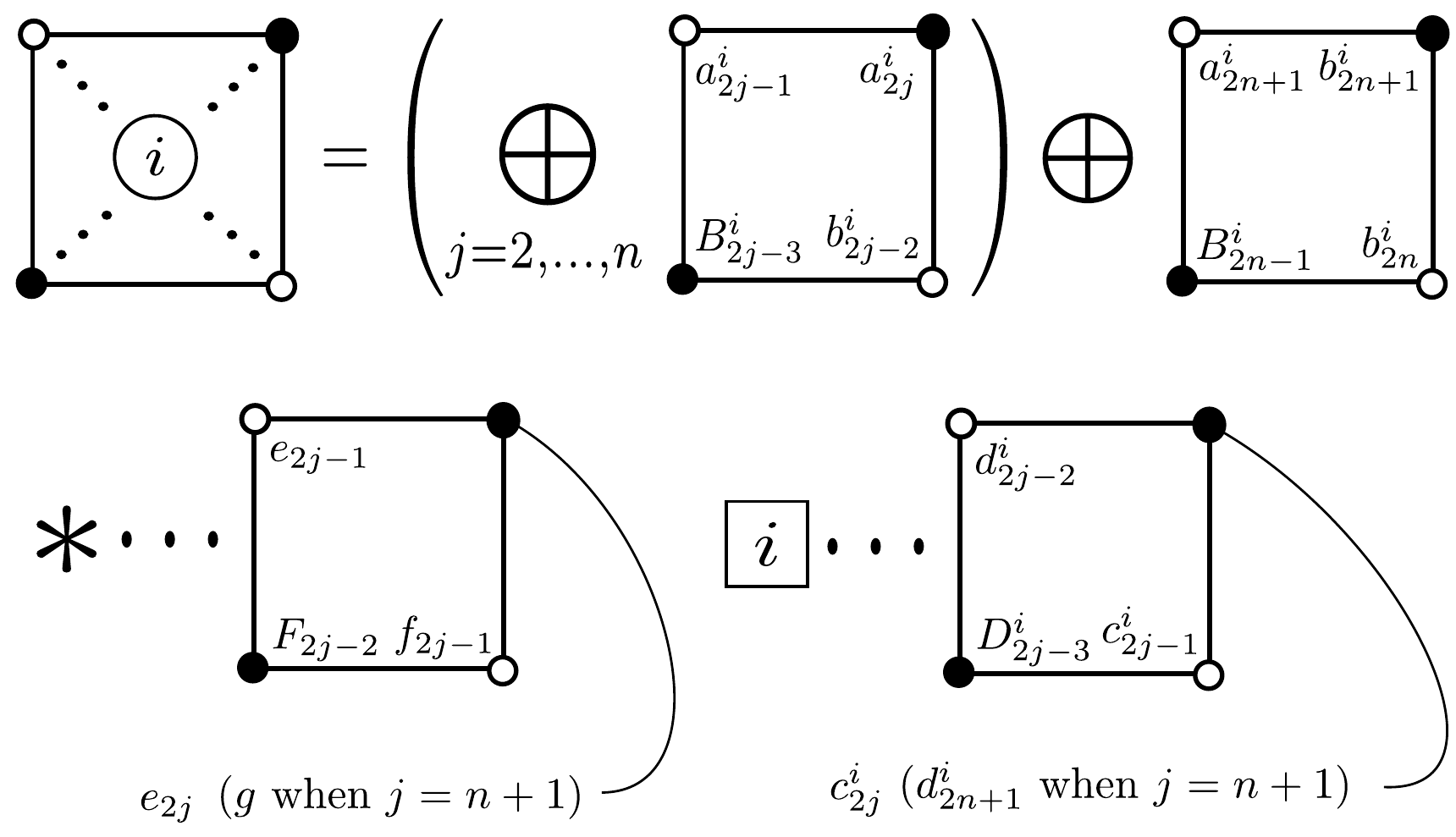}
\end{tabular}
\caption{The full knot Floer complex $\CFK^{\infty}(K_n^{(3,3k+2)})$. The box complex with \textcircled{\scriptsize $i$} represents the direct sum of $n$ box complexes as shown in the figure. The box with $*$ or a boxed number similarly represents the direct sum of $n$ box complexes (where $j$ runs $2,\ldots,n+1$). }
\label{CFK-K_n3,3k+2-precise}
\end{figure}

%%%%%%%%%%%%%%%%%%%%%%%%%%%%%%%%%%%%%%%%%%%%%%%%%%%%%%%%%%%%%%%%%%%%%%%%%%%%%%%%%%%%%%%%%%%%%%%%%%%%%%%%%%%%

\end{document}